\documentclass[11pt]{article}
\usepackage{epsfig}
\usepackage{amssymb,amsmath,amsthm,amscd}
\usepackage{latexsym}
 \usepackage{bbm,dsfont}

\newtheorem{thm}{Theorem}[section]
\newtheorem{prop}[thm]{Proposition}

\newtheorem{lem}[thm]{Lemma}

\theoremstyle{definition}
\newtheorem{defn}[thm]{Definition}

\newcommand{\be}{\begin{equation}}
\newcommand{\ee}{\end{equation}}

\newcommand{\R}{\mathbb{R}}
\newcommand{\N}{\mathbb{N}}
\newcommand{\E}{\mathbb{E}}

\def \eps {{ \varepsilon }}

\def \o {{  {\mathbb{R}^n} }}

\def \calf {{  {\mathcal{F}} }}

\def \cala {{  {\mathcal{A}}  }}

\def \cale {{  {\mathcal{E}}  }}
\def \calg {{  {\mathcal{G}}  }}

 \def \call {{  {\mathcal{L}}  }}

\begin{document}

\baselineskip=1 \baselineskip

%
%
%
%
%
%
%
%

\begin{center}
 {  \bf 
 Large Deviations  of Fractional Stochastic Equations
 with Non-Lipschitz Drift  and Multiplicative Noise
 on Unbounded Domains
   }
\end{center}

\medskip

\medskip

\begin{center}
 Bixiang Wang  
\vspace{1mm}\\
Department of Mathematics, New Mexico Institute of Mining and
Technology \vspace{1mm}\\ Socorro,  NM~87801, USA \vspace{3mm}\\
Email: bwang@nmt.edu\vspace{6mm}\\
\end{center}



\begin{abstract}  
This paper is concerned with the large deviation principle
of the non-local fractional stochastic reaction-diffusion equation
with a polynomial drift of arbitrary
degree driven by multiplicative
noise defined on unbounded domains.
We first prove the strong convergence of the solutions
of a control equation  with respect to the weak topology
of   controls, and then show the convergence
in distribution  of the solutions of the  stochastic
equation when the  noise intensity approaches zero. We finally
establish the large deviations of the stochastic equation by
the weak convergence method.
The main difficulty of the paper is caused by the non-compactness
of Sobolev embeddings on unbounded domains, and the idea of
uniform tail-ends estimates is employed to circumvent the obstacle
in order to obtain the tightness of distribution laws of the stochastic equation
 and the precompactness of  the control equation.
  \end{abstract}



\section{Introduction}
\setcounter{equation}{0}

In this paper, we
study the large deviation principle (LDP) of 
 the fractional stochastic reaction-diffusion
equation    with polynomial
drift    of arbitrary order and  Lipschitz
  diffusion 
  defined on the entire space $\R^n$. 
Given   $ \alpha \in (0,1)$,  consider the  non-local
fractional 
Ito stochastic 
equation for  $x\in \o$ and $t >0$:
\be
  \label{intr1}
  du^{\eps} (t)
  + (-\Delta)^ \alpha  u^{\eps} (t)  dt
  + F(t,x, u^{\eps}(t)) dt
  =  
    g(t,x)   dt
  +\sqrt{\eps} \sigma (t, x,  u^{\eps}(t))    {dW} ,
  \ee 
  with   initial condition 
 \be\label{intr3}
 u^{\eps}( 0, x ) = u_0 (x),   \quad x\in \R^n,
 \ee 
where 
$\eps \in (0,1)$ is the noise intensity,
 $(-\Delta)^\alpha$ is  the fractional Laplace operator,
 $F$ is a nonlinear function with polynomial
 growth of arbitrary  order in its third argument, 
  $g\in L^2_{loc}(\R; L^2(\o))$ is  given,
  $\sigma$  is a Lipschitz  nonlinear function,
     and $W$
 is a  
 cylindrical Wiener process
   in a   Hilbert
 space  defined 
 on a  complete filtered probability space.

 Fractional   partial  differential equations 
 have  numerous applications
    \cite{abe1,  garr1,    guan2,  
    jara1, kos1},
     and have been extensively
     studied  in the literature,
     see e.g.,    
   \cite{abe1,  caff1,  dine1,  gal1, garr1,  guan2,
    jara1, kos1,  luh1,   ros1, ser1, ser2} for the solutions 
    and    
     \cite{chen1, gu1, luh2,  wan10, wangJDE2019}
     for the dynamics of  such type of equations.
     In this  paper, we are interested in the LDP
     of the  
   the non-local
    fractional stochastic  equation \eqref{intr1}
   on $\R^n$ 
    as $\eps \to 0$.
    
     The  LDP   of
  stochastic  differential equations   has 
    been  well developed  for finite-dimensional
 as well as   infinite-dimensional  
 systems,
 which can be established by the classical method 
    \cite{bra1, card1, cerr1, che1,chow1,  dem1, 
  far1, fre1, fre2, gau1, kall1, mart1, pes1, sow1,
  str1,
  var1, var2, ver1},
  and the weak convergence method
   \cite{bess1, brz1, bud1, bud2, cerr2, cerr3,
   cerr4, chu1, 
   duan1, dup1, liu1, ort1, ren1, roc1,sal1, sal2}.
  The detailed description of the classical
  LDP  and
  the weak convergence theory can be found
  in \cite{fre1} amd \cite{bud1,  dup1},
  respectively.

    For the standard 
    stochastic reaction-diffusion 
    equation (i.e., $\alpha =1$ in \eqref{intr1}),
    the LDP has  been investigated in
    \cite{cerr5, che1, chow1, fre1, kall1,
    pes1, sow1}
    for  globally Lipschitz continuous or linearly growing nonlinearity $F$,
    and in \cite{cerr1, cerr3, hai1, ren1}
    for locally Lipschitz continuous $F$
    of  polynomial type.
    In particular,  the LDP of the equation
    driven by multiplicative noise  has been examined
    in \cite{cerr1}  and \cite{ren1} when $F$ is a polynomial of
    {\it any degree}.
    Note that in all these papers, the standard reaction-diffusion equation
    is defined in a {\it bounded } domain in $\R^n$,
    where the solution operators  of the corresponding deterministic equation
    are  compact due to the compactness  of Sobolev embeddings
    on bounded domains. The compactness of the solution operators
    as well as   Sobolev embeddings play a key role
    for proving the LDP of the equation on bounded domains.
    Indeed,
     such compactness is  a main avenue to
    establish  the tightness and   convergence of a family of solutions
   of the stochastic equation on  bounded  
   domains  as $\eps \to 0$, especially when $F$ is a polynomial of
   arbitrary degree (see \cite{cerr1,ren1}).
    
    However, for unbounded domains like $\R^n$,
    neither  Sobolev embeddings  nor the solution operators
    of the equation are compact, and hence the methods for 
    proving the  LDP of the equation in bounded domains
    do not apply to the case of unbounded domains.
    The non-compactness of Sobolev embeddings on unbounded
    domains is actually a major obstacle for studying the 
    dynamics and the long term
    behavior of solutions of stochastic partial differential equations
    defined on unbounded domains,  
    which include the existence
    and stability  of invariant measures
    and random attractors.

    As far as  the author is aware,  there is no result available
    in the literature on the LDP of   stochastic 
    parabolic equations with  a superlinear drift
     defined on unbounded domains.
     The purpose of the present paper is to  solve this problem
     and establish the LDP of the stochastic equation \eqref{intr1}
     with a polynomial drift of any degree
     defined on the entire space $\R^n$.

     To that end, we will adopt the idea of uniform tail-ends
     estimates of solutions to
     circumvent the difficulty caused by the non-compactness
     of Sobolev embeddings   on $\R^n$.
     More precisely, we will show the solutions of the control equation
     for \eqref{intr1} are uniformly small over a finite time interval
     when the space variable is sufficiently large (see Lemma \ref{tail}).
     Such uniform smallness of solutions in far field   
      combined with the compactness of     embeddings in bounded domains
      will allow us to prove the tightness of a family of solutions
      of the stochastic equation and the precompactness of a family
      of solutions of the control equation
      defined on unbounded domains  when the controls vary over a
      bounded set (see Lemma \ref{wc_sol}).
      The precompactness of solutions of the control equation
      is   essential for proving the convergence in distribution
       of the solutions of the
      stochastic equation  as $\eps \to 0$
      (see Lemma \ref{cso}) , which in turn is the key
      for establishing the LDP of the stochastic equation
      (see Theorems \ref{main} and \ref{main1}) .

      Note that 
      the stochastic reaction-diffusion equation
      \eqref{intr1} is a non-local fractional equation
      with order $\alpha \in (0,1)$, and the LDP of such
      fractional equations has not been reported in the literature
      even if the domain is bounded.
      It seems  that the present paper
      is the first one to deal with the LDP of fractional equations
      on unbounded domains.
      We  point out that  the results of this paper 
      are also valid for the standard reaction-diffusion
     when $\alpha=1$ 
     in both bounded and unbounded domains.
     Actually, the proof of the LDP  
     for  $\alpha =1$ is simpler than the case $\alpha\in (0,1)$.

      We remark that the idea of uniform tail-ends estimates
      was initially introduced in \cite{wangPD1999} to prove
      the  existence of global attractors  of deterministic equations, 
      and this is the first time such tail-estimates
      are employed to 
      show  the LDP of stochastic 
      partial differential equations defined on unbounded domains.
      The approach of this paper  can be applied to a broad class
      of  stochastic   differential  equations such the hyperbolic equations and
      the functional equations, which 
      we plan to further investigate in the future,

       It is worth mentioning that if the nonlinear drift 
       has a linear growth rate, then the LDP of the stochastic equations
       can be established by the 
       standard weak convergence method even if
       the domain is unbounded,  see, e.g., \cite{roc1}.
       However, in the present paper, we deal with a  polynomial drift
       of arbitrary order, and  the standard method does not apply
       in this case. This is why we must appeal to the 
       approach  of the uniform tail-ends estimates to handle
       the tightness and precompactness of solutions on unbounded domains.

     In the next section, we borrow basic results 
     on   the weak 
     convergence method for LDPs.
     In Section 3, we discuss 
      the existence and uniqueness
     of solutions of \eqref{intr1}-\eqref{intr3}
    under certain conditions. 
     Section 4 is devoted to the LDP
     of \eqref{intr1}-\eqref{intr3}
     in   $C([0,T], L^2(\R^n))
     \bigcap L^2(0,T; H^\alpha (\R^n ))$.
     In the last section, we prove 
     the LDP
     of \eqref{intr1}-\eqref{intr3}
     in   
       $C([0,T], L^2(\R^n))
     \bigcap L^2(0,T; H^\alpha (\R^n))
     \bigcap L^p(0,T; L^p(\R^n))$
     under further  dissipative assumptions.

 \section{Weak convergence method 
 for  large deviations}
 
 In this section, we recall the
 weak convergence  method  for the LDP
   of a family
 of random variables
 that was   introduced
 in \cite{bud1, dup1}.

   Let $(\Omega, \mathcal{F}, 
   \{ { \mathcal{F}} _t\} _{t\ge 0},  P )$
be a  complete filtered probability space,
and $l^2$ be the Hilbert space
of square summable   sequences of real numbers.
Assume  $\{W(t)\}_{t\in [0,T]} $
with $T>0$  is 
a cylindrical Wiener process
with identity covariance operator 
in   $l^2$
with respect to 
$(\Omega, \mathcal{F}, 
   \{ { \mathcal{F}} _t\} _{t\ge 0},  P )$,
   which means
   that there exists a   separable 
   Hilbert space $U$
   such that
    the embedding
  $l^2 \hookrightarrow U$ is   Hilbert-Schmidt
  and $W(t)$ takes values in $U$.

Suppose  $\cale$ 
is  a polish space,   
 $\calg^\eps: C([0,T], U) \to \cale$ 
 is  a
 measurable map, and 
 $
 X^\eps = \calg ^\eps (W)$
 for every
 $\eps>0$.
 For every  $N>0$,  define
  \be\label{pre0a}
S_N
=\{v\in L^2(0,T; l^2): \int_0^T \| v(t)\|_{l^2}^2 dt
\le N\}.
\ee
Then $S_N$ is a polish space
endowed with the weak topology, which is assumed
throughout the paper.
Let $\cala$
 be the space of all
 $ l^2$-valued stochastic processes
 $v$
 which are progressively measurable
 with respect to $\{\calf_t\}_{t\in [0,T]}$
 and
 $\int_0^T \| v (t)\|^2 dt<\infty$
 $P$-almost surely.
 Let $ \cala_N$ be a subset of $\cala$ as given by
  \be\label{pre0b}
 \cala_N
 =\{v\in \cala: v (\omega) \in S_N
 \ \text{for almost  all } \omega \in \Omega
 \}.
\ee

 \begin{defn}
 A function $I : \cale \to [0, \infty]$
 is called a rate function
 on $\cale$  if
 it is lower  semi-continuous in $\cale$.
 A rate function $I$ on $\cale$ is 
 said to be a good
 rate function on $\cale$ if for every
 $0\le C <\infty $, the level set
 $\{x\in \cale: I(x) \le C\}$ is a compact
 subset of $\cale$.
 \end{defn}
 
 \begin{defn}
 The family $\{X^\eps\}$ is said to satisfy the
LDP
 in $\cale$ with a
 rate function $I:\cale \to [0,\infty]$ if 
 for every  Borel subset $B$ of $\cale$,
 $$
 -\inf_{x\in B^\circ}
  I(x)
  \le \liminf _{\eps \to 0}
  \eps \log 
  P(X^\eps \in B )
  \le
     \limsup  _{\eps \to 0}
  \eps \log 
  P(X^\eps \in B )
  \le
   -\inf_{x\in \overline{B} }
  I(x).
  $$
  where $B^\circ$ and $\overline{B}$ are
  the interior and the closure of $B$ in $\cale$,
  respectively.
 \end{defn}

    To prove the LDP 
   of $X^\eps$ by the weak convergence method, 
   the following conditions are needed:  
  there exists a measurable map
  $\calg^0:    C([0,T], U) \to \cale$
  such that
  \begin{enumerate}
  \item[(\bf {H1})] \    If $N<\infty$ and
  $\{v^\eps\}\subseteq \cala_N$
  such that $\{v^\eps\}$
  converges  in distribution
  to $v$ as 
  $S_N$-valued random variables,
  then
  $\calg^\eps \left (
  W +\eps^{-\frac 12} \int_0^{\cdot}
  v^\eps (t) dt
  \right )$
  converges in distribution
  to 
    $\calg^0 \left (
    \int_0^{\cdot}
  v  (t) dt
  \right )$.
  
  \item[(\bf {H2})]   \  For every $N<\infty$, the
  set  $\left \{
  \calg^0 (\int_0^{\cdot} v(t) dt):\
  v \in S_N
  \right \}$
  is a compact subset of $\cale$.
    \end{enumerate}

  Define $I: \cale \to [0, \infty]$ by,
  for every $x\in \cale$,
 \be\label{pre1}
  I(x)= \inf 
  \left \{
  {\frac 12} \int_0^T \| v(t)\|^2_{l^2} dt:\
   v\in L^2(0,T;  {l^2})
   \ \text{such that}\  
  \calg^0 \left (\int_0^\cdot v(t) dt\right ) =x
  \right \},
 \ee
  with the convention that
  the infimum over an  empty set
   is   $\infty$.

  It is known,     by   ${\bf (H2)}$, 
 the  map $I$ is a good
  rate function on $\cale$.
  Moreover,  ${\bf (H1)}$ and
  ${\bf (H2)}$ are sufficient for 
  the family $\{X^\eps\}$  to satisfy the 
  LDP  in $\cale$ with rate function $I$
  (see, \cite[Theorem 4.4]{bud1}).
  
  \begin{prop}\label{LP1}
  If $\{\calg^\eps\}$ satisfies
    {\rm {({\bf H1})-({\bf H2})}}, then
    the family
    $\{X^\eps\}$ 
    satisfies  
    the LDP  in $\cale$ with rate function $I$
  as defined by \eqref{pre1}.
  \end{prop}

\section{Existence and uniqueness of solutions} 
\setcounter{equation}{0}

In this section, we  first describe  the assumptions on the
nonlinear terms in \eqref{intr1} and then discuss 
the 
existence and uniqueness of solutions of 
\eqref{intr1}-\eqref{intr3}.

Let $\mathcal{S}$ be the Schwartz space 
of rapidly decaying 
$C^\infty$ functions on $\R^n$.  Then
 the  fractional Laplace  operator
 $(-\Delta )^\alpha $ for $0<\alpha<1$ is  defined  by,
  for $u\in \mathcal{S}$,
$$
(-\Delta)^\alpha  u (x)
=- {\frac 12} C(n,\alpha) 
\int_{\R^n}
{ \frac {u(x+y) + u(x-y) -2 u(x)}{|y|^{n+2\alpha}}     } dy,
\quad x \in \R^n,
$$
where  $C(n,\alpha) 
 =\frac{\alpha 4^\alpha\Gamma(\frac{n+2\alpha}{2})}
{\pi^\frac{n}{2}\Gamma(1-\alpha)}
$.
 By  \cite{dine1}
we find that 
for any $u \in {\mathcal{S}}$,
\be\label{del2}
(-\Delta )^\alpha u
= 
  {\mathcal{F}^{-1}}
(|\xi|^{2 \alpha} (\mathcal{F} u)), \quad \xi \in \R^n,
\ee
where
${\mathcal{F}}$    and 
${\mathcal{F}}^{-1}$ are 
the Fourier transform   and the 
inverse 
 Fourier transform, respectively.
 For $0< \alpha<1$,  the  fractional Sobolev space
    $H^ \alpha (\R^n)$  is defined by
 $$
 H^ \alpha (\R^n)
 =\left \{
 u\in L^2(\R^n):
 \int_{\R^n} \int_{\R^n}
 {\frac {|u(x)- u(y)|^2} {|x-y|^{n+2 \alpha} }} dxdy <\infty
 \right \},
 $$
 endowed with   norm
$$
 \| u\|_{H^ \alpha ({\R^n}) }
 =
 \left (
 \int_{\R^n} |u(x)|^2 dx 
  +
  {\frac {1}2}C(n,\alpha)
  \int_{\R^n} \int_{\R^n}
 {\frac {|u(x)- u(y)|^2} {|x-y|^{n+2 \alpha} }} dxdy
 \right )^{\frac 12},
\   \forall    u \in H^ \alpha (\R^n),
$$
and inner product  
 $$
 (u, v)_{H^ \alpha (\R^n)}
 =
  \int_{\R^n} u(x)v(x)  dx 
  +  {\frac {1}2}C(n,\alpha)
  \int_{\R^n} \int_{\R^n}
 {\frac {(u(x)- u(y)) (v(x) -v(y))} {|x-y|^{n+2 \alpha} }} dxdy,
 \ \forall   u, \ v \in H^ \alpha (\R^n).
 $$

It follows from \cite{dine1} 
that the norm and the  inner product of
$H^ \alpha  ({\R^n})$ can be calculated as 
follows: for all $u, v \in H^ \alpha  ({\R^n})$, 
$$
 \| u\|_{H^ \alpha (\R^n)}^2
 = \| u\|^2_{L^2(\R^n)}
 +   \| (-\Delta)^{{\frac  \alpha2}} u \|^2_{L^2(\R^n)}, 
$$
and 
 $$
 (u, v)_{H^ \alpha (\R^n)}
 =
  \int_{\R^n} u(x)v(x)  dx 
  +
  \int_{\R^n}
  \left ( (-\Delta)^{{\frac  \alpha2}} u (x)
  \right )
  \left (
    (-\Delta)^{{\frac  \alpha2}} v (x)
    \right ) dx.
    $$

 In the sequel,    we  will write
$H= 
  L^2(\R^n) $  and 
$V= 
H^ \alpha (\R^n) $.
Then we have
$V \hookrightarrow H  = H^* \hookrightarrow V^*$
where $H^*$   and $V^*$ are the dual spaces
of $H$  and $V$, respectively, and $H^*$ is identified with
$H$ by Riesz\rq{}s representation  theorem.
The norm and the inner product of $H$ are denoted by
   $\| \cdot \|$  and $(\cdot, \cdot)$, respectively.
   We  also use $\call_2(H_1,H_2)$   for the space of Hilbert-Schmidt
   operators  from  
   a 
   separable 
   Hilbert space  $H_1$ to another
   separable
    Hilbert space  $H_2$
   with norm $\| \cdot \|_{\call_2(H_1,H_2)}$.

For the nonlinear drift  
$F$ in  \eqref{intr1}, 
we  assume  that 
$F: \R \times \o \times \R$
$\to \R$ is      continuous     
 such that  for all
$t, u, u_1, u_2  \in \R$   and  $x \in \o$, 
\be 
\label{F1}
 F(t,x,0) =0,
 \ee
\be 
\label{F2}
F (t, x, u) u
\ge \lambda_1 |u|^p -\psi_1(t,x),
\ee
\be 
\label{F3}
|F(t, x, u_1)- F(t, x, u_2) |   \le
\lambda_2 \left (\psi_2 (t,x)
+
 |u_1|^{p-2}  + |u_2|^{p-2}
 \right ) | u_1-u_2|,
\ee
\be 
\label{F4}
{\frac {\partial F}{\partial u}} (t, x, u)  \ge -  \psi_3(t,x) ,
\ee
where  $\lambda_1>0$,
  $\lambda_2>0$,
   $ p\ge  2$, 
$\psi_1 \in L^1_{loc} (\R, L^1(\R^n))$,
$\psi_2, \psi_3 \in L_{loc}^\infty (\R, L^\infty(\R^n))
\bigcap  L_{loc}^1(\R, L^1(\R^n))$.

Since the case  $p=2$ is simpler to deal with,
from now on, we assume $p>2$.
 It follows from \eqref{F1}  and \eqref{F3}
that there exists $\lambda_3>0$  and
$\psi_4\in L_{loc}^\infty (\R, L^\infty(\R^n))
\bigcap  L_{loc}^1(\R, L^1(\R^n))$ such that
 for   $t, u \in \R$ and   $x\in \o$,
 \be 
\label{F5}
|F(t, x, u) |   \le
\lambda_3  |u|^{p-1}  + \psi_4 (t,x).
\ee
 In addition, by \eqref{F3} we have
 for   $t, u \in \R$ and   $x\in \o$,
 \be 
\label{F6}
|{\frac {\partial F}{\partial u}}
 (t, x, u )  |   \le
\lambda_2 \left (\psi_2 (t,x)
+
 2 |u |^{p-2}   
 \right ).
\ee

   For the diffusion term
    in \eqref{intr1},  we assume
    that
        $\sigma:
      \R \times \R^n \times \R
      \to l^2$ is a map given by
  \be\label{sig0}
     \sigma (t,x, s)
     = \sigma_1  (t,x)
     +   
       \kappa (x)   \sigma_{2} ( t,x, s)  
        , 
       \quad \forall \ t\in \R, \ x\in \R^n, \ s\in \R,
\ee
     where     
   $\kappa \in    L^2(\R^n) \bigcap L^\infty(\R^n)$,
   $\sigma_1: \R \to L^2(\R^n, l^2 )$ is continuous,
   and $\sigma_2(t,x,s)$ is Lipschitz continuous  
   in $l^2$ 
   with respect to $s\in \R$ uniformly
   in $t\in  \R$  and $x\in \R^n$.
   More precisely, we  assume that
    $ 
     \sigma_2 (t,x, s)
     =\left \{ 
        \sigma_{2,k} ( t,x, s)  \right \}_{k=1}^\infty
      $  satisfies the condition:   
    for every $k\in \N$,
   there exist  positive numbers $\alpha_k$, $\beta_k$ and $\gamma_k$
   such that   for all 
   $t\in \R$,
   $x\in \R^n$  and $s, s_1, s_2 \in \R$,
   \be\label{sig1}
   |\sigma_{2,k} (t,x, s_1) -\sigma_{2,k} (t,x,s_2) |
   \le \alpha_k |s_1  -s_2 |,
   \ee
   and
   \be\label{sig2}
   |\sigma_{2,k}  (t,x,s)  |
   \le \beta_k    +\gamma _k  |s|,
   \ee
   where
  \be\label{sig3}
  \sum_{k=1}^\infty  ( \alpha_k^2
   +\beta_k^2 + \gamma_k^2  ) <\infty .
     \ee
     For convenience,
     we write  
     $\sigma_1 (t,x)
     =\left \{ \sigma_{1,k} (t,x)\right \}_{k=1}^\infty
     \in l^2$
     for $t\in \R$ and  $x\in \R^n$.
     Then we have
     $$
     \sigma   (t,x,s) 
     = 
     \left \{ \sigma_{1,k}  (t,x)
     +   
       \kappa (x)   \sigma_{2,k} ( t,x, s)  
       \right \}_{k=1}^\infty
        , 
       \quad \forall \ t\in \R, \ x\in \R^n, \ s\in \R.
  $$
  For every $u\in H$ and $t\in \R$, 
  by \eqref{sig2} and \eqref{sig3} we see that
  $\sigma (t,\cdot, u) \in L^2(\R^n, l^2)$ and
  $$
  \| \sigma (t, \cdot, u)\|^2_{L^2(\R^n, l^2)}
  =\sum_{k=1}^\infty
  \int_{\R^n}
  \left |  \sigma_{1,k}  (t,x)
     +   
       \kappa (x)   \sigma_{2,k} ( t,x, u(x) ) 
       \right |^2 dx
  $$
  \be\label{sig4}
  \le
  2 \sum_{k=1}^\infty 
  \|\sigma_{1, k} (t) \|^2
  + 4\| \kappa \|^2
  \sum_{k=1}^\infty \beta_k^2
  +4 \|\kappa\|^2_{L^\infty (\R^n)}
  \| u \|^2 \sum_{k=1}^\infty
  \gamma_k^2.
  \ee
  
  Given $u\in H$  and $t\in \R$, define an operator
  $\sigma (t, u): l^2\to H$ by
\be\label{sig5}
  \sigma (t,u) (v) (x)
  =\sum_{k=1}^\infty
  \left (
   \sigma_{1,k}  (t,x)
     +   
       \kappa (x)   \sigma_{2,k} ( t,x, u(x) )
       \right )
         v_k,
  \quad \forall \ 
   v=\{v_k\}_{k=1}^\infty \in l^2,
   \ \ x\in \R^n.
\ee
By \eqref{sig4}  and \eqref{sig5} we have
for all      $u\in H$,  $t\in \R$
and $v\in l^2$,
$$
 \|  \sigma (t,u) (v)  \|^2
 \le
  \| \sigma (t, \cdot, u)\|^2_{L^2(\R^n, l^2)}
  \| v\|^2_{l^2},
$$
which shows that
   $\sigma (t, u): l^2\to H$ is a linear bounded operator.
   Actually, this operator is  Hilbert-Schmidt 
   with norm
 $$
 \|  \sigma (t,u)   \|^2_{\call_2 (l^2, H)}
 =
   \| \sigma (t, \cdot, u)\|^2_{L^2(\R^n, l^2)},
   $$
  which  along with \eqref{sig4}
  shows that 
  for all      $u\in H$  and   $t\in \R$,
   \be\label{sig6}
 \|  \sigma (t,u)   \|^2_{\call_2 (l^2, H)}
 \le
 L_1 (1+  \|u \|^2 )
  + 2 \sum_{k=1}^\infty 
  \|\sigma_{1, k} (t) \|^2,
  \ee
  where $L_1= 
  4\| \kappa \|^2
  \sum\limits_{k=1}^\infty \beta_k^2 +
  4 \|\kappa\|^2_{L^\infty (\R^n)}
   \sum\limits_{k=1}^\infty
  \gamma_k^2$.

 On the other hand, by \eqref{sig1} and \eqref{sig5}
  we have for all $u_1, u_2
  \in H$,
  $$
  \| \sigma (t,u_1) -\sigma (t, u_2)\|^2_{\call_2 (l^2, H)}
  =\sum_{k=1}^\infty
  \int_{\R^n}
 \kappa^2  (x) \left |
 \sigma_{2,k} (t,x, u_1(x))
 -
 \sigma_{2,k} (t,x, u_2(x))
 \right |^2 dx
 $$
 $$
 \le
 \sum_{k=1}^\infty
 \alpha_k^2  \int_{\R^n}
 \kappa^2  (x) \left |
  u_1(x) 
 -  u_2(x) 
 \right |^2 dx,
 $$
  which shows that
  for all $t\in \R$,
  \be\label{sig7}
  \| \sigma (t,u_1) -\sigma (t, u_2)\|^2_{\call_2 (l^2, H)}
 \le
  \| \kappa \|^2_{L^\infty(\R^n)}
  \| u_1-u_2\|^2
  \sum_{k=1}^\infty
 \alpha_k^2,
 \quad \forall \   u_1, u_2\in H.
 \ee

%
%
%

  With above notation,
  system \eqref{intr1}-\eqref{intr3}
  can be rewritten as 
  \be
  \label{sys1}
  du^{\eps} (t)
  + (-\Delta)^ \alpha  u^{\eps} (t)  dt
  + F(t,x, u^{\eps}(t)) dt
  $$
  $$
  =  
    g(t,x)   dt
  +\sqrt{\eps} 
  \sum_{k=1}^\infty
  \left (
  \sigma_{1,k}  (t,x)
     +   
       \kappa (x)   \sigma_{2,k} ( t,x, u^{\eps}(t) ) 
       \right ) dW_k
   ,
  \ee 
  with   initial condition 
 \be\label{sys2}
 u^{\eps}( 0, x ) = u_0 (x),   \quad x\in \R^n,
 \ee 
  where $\{W_k\}_{k=1}^\infty$
  is a sequence of independent  real-valued
  standard 
  Wiener processes.

A solution of system \eqref{intr1}-\eqref{intr3}
(or system \eqref{sys1}-\eqref{sys2}) is understood
in the following sense.

 \begin{defn}
 \label{defnsol}
Given  $T>0$,
$\eps\in (0,1)$
  and   $u_0 \in L^2(\Omega, \calf_0; H)$,
   a continuous $H$-valued $\calf_t$-adapted
 stochastic process
 $u^{\eps}$ is  called a solution of \eqref{intr1}-\eqref{intr3}
 on $[0,T]$
 if
\be\label{defnsol_1}
 u^{\eps}\in L^2(\Omega, C([0, T], H))
 \bigcap L^2(\Omega, L^2(0, T; V))
 \bigcap L^p(\Omega, L^p(0, T; L^p(\R^n))),
\ee
   such that   for all  $t\in [0,T]$
 and $\xi \in V\bigcap L^p(\R^n)$,
\be\label{defnsol_2}
 ( u^{\eps}(t), \xi )
 + \int_0^t
 ( (-\Delta )^{\frac \alpha 2} u^{\eps}(s) , (-\Delta )^{\frac \alpha 2} \xi) ds
 +
 \int_0^t\int_\o F(s, x, u^{\eps}(s)) \xi (x) dx  ds
 $$
 $$
 =(u_0, \xi)
    +\int_0^t (g(s),  \xi  ) ds
  +\sqrt{\eps} \int_0^t
  ( \xi , 
  \sigma (s, u^{\eps}(s)) dW )
\ee
 $P$-almost surely.
 \end{defn}
 
 It follows from \eqref{F5}  and \eqref{defnsol_1}
 that  if $u^{\eps}$ is a solution of \eqref{intr1}-\eqref{intr3},
  then 
 $$
 (-\Delta)^\alpha u^{\eps} \in  L^2(\Omega, L^2(0,T; V^*) )
  \ \text{  and } \
  F\in  L^q(\Omega, L^q(0,T; L^q(\R^n))),
  $$
  where $p$  and $q$ are conjugate exponents;
  that is, ${\frac 1q} +{\frac 1p} =1$.
   As such, $u^{\eps}$ satisfies  
\eqref{defnsol_2} if and only if 
  for all  $t\in [0,T]$,  
\be\label{defnsol_3}
   u^{\eps}(t) 
 + \int_0^t
   (-\Delta )^{  \alpha  } u^{\eps}(s)  ds
 +
 \int_0^t  F(s, \cdot , u^{\eps}(s))    ds
 =u_0
   +\int_0^t  g(s)  ds
  + \sqrt{\eps} \int_0^t
  \sigma (s, u^{\eps}(s)) dW
\ee
 in $ (V\bigcap L^p(\R^n) )^*$, $P$-almost surely.

  Under conditions
  \eqref{F1}-\eqref{F4}
  and \eqref{sig0}-\eqref{sig3},
  it is known that problem \eqref{intr1}-\eqref{intr3}
  possesses a unique solution
  in the sense of Definition \ref{defnsol}
  (see \cite{wangJDE2019}),  which is stated below.

 \begin{prop} \label{posed1}
 Suppose   \eqref{F1}-\eqref{F4}
  and \eqref{sig0}-\eqref{sig3}  hold.
  If   $T>0$,
  $\eps\in (0,1)$
  and  $u_0\in L^2(\Omega,\calf_0; H)$,
  then   problem \eqref{intr1}-\eqref{intr3}
 has a unique solution $u^{\eps}$ which 
   satisfies 
   the energy equation: 
      for all    $t\in [0, T]$, 
           $$
       \| {u}^{\eps}  (t) \|^2
       + 2\int_0^t \| (-\Delta)^{\frac \alpha{2}} {u}^{\eps}  (s)\|^2
       ds
       + 2\int_0^t \int_{\R^n}
        F (s, x, u^{\eps}(s) )   {u}^{\eps}  (s)    dx ds
       $$
       $$
       = \|u_{0} \|^2 
                + 2 \int_0^t (u^{\eps} (s) , g (s) )ds
        +
       2\sqrt{\eps}
       \int_0^t 
     ( {u}^{\eps} (s) , 
        \sigma (s, u^{\eps}(s) )  dW )
       + \eps  \int_0^t \|  \sigma (s, u^{\eps}(s) )  \|^2_{\call_2(l^2,H)}
       ds,
    $$  $P$-almost surely. 
    Moreover,  the following   uniform estimates
    are valid:
 $$
\| u^{\eps}  \|_{L^2(\Omega, C([0, T], H) ) }^2
+
\| u ^{\eps}  \|_{L^2 (\Omega, L^2(0, T;  V ) )}^2
+
\| u ^{\eps} \| _{L^p(\Omega, L^p(0, T;  L^p(\R^n) ) )}^p
$$
$$
\le M_1
\left (1+ 
\| u_0 \|^2_{L^2(\Omega, H)}
+ \| g\|^2_{L^2(0,T; H)}
\right ),
$$
where $M_1=M_1(T) $ is a positive    number
depending on $T$ but not on
 $u_0$ or $\eps$.
 \end{prop}

  In the rest of the paper, we will investigate
  the LDP of the family $\{u^\eps\}_{\eps>0}$
  of the solutions of \eqref{intr1}-\eqref{intr3}
  in the space 
     $
     C([0, T], H )\bigcap  L^2(0, T;  V  )
     \bigcap  L^p(0, T;  L^p(\R^n)  )$
     as $\eps \to 0$.

   \section{Large deviation principles of stochastic equations} 
\setcounter{equation}{0}

This section is devoted to  the 
LDP of  problem     \eqref{intr1}-\eqref{intr3}
in the space  $
     C([0, T], H )\bigcap  L^2(0, T;  V  )
    $
     as $\eps \to 0$, which is 
   stated below.

 \begin{thm}\label{main}
  Suppose   \eqref{F1}-\eqref{F4}
  and \eqref{sig0}-\eqref{sig3}  hold, and $u^\eps$
 is the solution of 
       \eqref{intr1}-\eqref{intr3}.
   Then the family
   $\{u^\eps\}$,  as $\eps \to 0$, 
satisfies the LDP
in   $
     C([0, T], H )\bigcap  L^2(0, T;  V  )
     $
      with  good rate function
as given by \eqref{rate_LS}.
\end{thm}

      Next, we prepare to prove  
 Theorem \ref{main}. 
  It follows from Proposition 
   \ref{posed1}
  that  for every  $\eps\in (0,1)$,
  there exists a Borel-measurable map
  $\calg^\eps: C([0,T], U) \to  C([0, T], H )\bigcap  L^2(0, T;  V  )
     \bigcap  L^p(0, T;  L^p(\R^n)  )$
  such that the solution $u^\eps$ of
  \eqref{intr1}-\eqref{intr3} is given by
  $$
  u^\eps =\calg^\eps (W),
  \quad  \text{P-almost surely}.
  $$
  
  As usual, to
   study the  
LDP  of
  $\{u^\eps\}$ as $\eps \to 0$, we
 need to consider the  deterministic  control
  system corresponding to \eqref{intr1}.
  Given a control  $v\in L^2(0,T; l^2)$, 
  find  $u_v$  by  the controlled  equation:
   \be\label{contr1}
{\frac {d u_v (t)}{dt}}
+  (-\Delta)^\alpha u_v  (t)  
+F(t, \cdot, u_v (t) )=  
  g (t)  
 +   \sigma(t, u_v (t))  
    v (t),
\ee
 with  initial data
 \be\label{contr2}
 u_v  (0)=u_0 \in H.
 \ee
  
  A solution of \eqref{contr1}-\eqref{contr2}
  is understood  in the sense  similar 
  to  Definition \ref{defnsol}.
  Of course,    \eqref{contr1} 
  is a deterministic equation rather than a
  stochastic one as \eqref{intr1}.
  The following lemma is concerned with
 the existence and uniqueness  
  of solutions  to
 \eqref{contr1}-\eqref{contr2}.

 \begin{lem}\label{exis_sol}
 Suppose   \eqref{F1}-\eqref{F4}
  and \eqref{sig0}-\eqref{sig3}  hold.
   If 
  $v\in L^2(0,T; l^2)$,
   then problem
\eqref{contr1}-\eqref{contr2}
 has a unique solution
 $u_v \in C([0,T], H)\bigcap L^2(0,T;V)
 \bigcap L^p(0,T; L^p(\R^n))$.

 Furthermore,
     for each $R_1>0$  and $R_2>0$,
    there  exists
       $M_2 =M_2  (R_1,R_2, T)>0$ such that  
       for any $u_0, u_{0,1}, u_{0,2}
       \in  H $ with 
          $\| u_{0} \|\le R_1$,
    $\| u_{0,1} \|\le R_1,
    \| u_{0,2}  \|\le R_1$,  and any
    $v, v_1, v_2\in L^2(0, T; l^2)$
    with
    $\| v\|_{L^2(0, T; l^2)}\le R_2$,
    $\| v_1\|_{L^2(0, T; l^2)}\le R_2$
    and $\| v_2\|_{L^2(0, T; l^2)}\le R_2$,
    the solutions
     $u_{v}$,
     $u_{v_1}$ and $u_{v_2}$
    of  \eqref{contr1}-\eqref{contr2}
    with initial data
    $u_0$,  $u_{0,1}$
    and $u_{0,2}$,   respectively,   
    satisfy  for all $t\in [0,T]$,
    \be\label{exis_sol 1}
    \| u _{v_1} (t) -u _{v_2} (t) \| ^2
    +   \int_0^T \| u_{v_1} 
    (t) -u_{v_2} (t) \|^2_{V}   dt
    \le
   M_2
    \left (
    \| u_{0,1}- u_{0,2}\|^2
    +\|  v_1- v_2 \|^2_{L^2(0, T; l^2)}
    \right ),
   \ee
   and
    \be\label{exis_sol 2}
    \| u_{v} (t) \| ^2
    +
      \int_0^T \|  u_v(t) \|^2_{V}
   dt 
   + 
   \int_0^T
   \| u_v(t)\|^p_{L^p(\R^n)} dt
    \le
   M_2 .
   \ee
 \end{lem}

  \begin{proof}
  Given $T>0$, 
   $v\in L^2(0,T; l^2)$ 
   and $u_0\in H$,
   under conditions
    \eqref{F1}-\eqref{F4}
  and \eqref{sig0}-\eqref{sig3},
   one may verify that
   problem \eqref{contr1}-\eqref{contr2}
   has at least one solution  
     defined
     on $[0,T]$ by the standard argument, see, e.g.,
     \cite{wangJDE2019}. 
     Next, we show \eqref{exis_sol 1},
     which also implies the uniqueness of solutions.

     For convenience, we now write
    $u_v$ as $u$. Then by   \eqref{contr1},
    we get 
   for all $t\in (0, T)$,
    $$
   {\frac 12} {\frac d{dt}}
   \| u(t) \|^2
   =-\| (-\Delta)^{\frac {\alpha}2} u(t)\|^2
    -\int_{\R^n} F(t,x, u(t))  u(t)  dx 
   $$
   \be\label{exis_solp1}
   + (g(t), u(t))
   +(\sigma (t, u(t)) v(t), u(t)).
  \ee
  We now deal with the right-hand side of \eqref{exis_solp1}.
  By  \eqref{F2} we have
    \be\label{exis_solp2}
   - \int_{\R^n} F(t,x, u(t))  u(t)  dx 
  \le  -\lambda_1 \| u(t)\|^p_{L^p(\R^n)}
  +  \|\psi_1 (t)\|_{L^1(\R^n)}.
  \ee 
    By Young\rq{}s inequality we have
 \be\label{exis_solp4} 
   |(g(t), u(t))|
  \le {\frac 12}  \| u(t) \|^2 +
  {\frac 12}
   \| g(t) \|^2.
  \ee
  For the last term
  on the right-hand side of \eqref{exis_solp1},
  by \eqref{sig6}  we get
  for  $t\in [0,T]$,
 $$
   | (\sigma (t, u(t)) v(t), u(t))| 
  \le {\frac 12}
  \| \sigma (t, u(t))\|^2_{\call_2(l^2, H)} \| v(t)\|^2_{l^2}
  +{\frac 12} \| u(t) \|^2
  $$
  $$
  \le {\frac 12} \left ( 1
+ L_1 \| v(t) \|^2_{l^2}
\right )   \|u(t) \|^2
+ \left (
{\frac 12} L_1 + 
\sum_{k=1}^\infty \| \sigma_{1,k} (t)\|^2
\right )
  \| v(t)\|^2_{l^2} 
  $$ 
    \be\label{exis_solp5}
  \le  {\frac 12}
   \left ( 1
+ L_1 \| v(t) \|^2_{l^2}
\right )   \|u(t) \|^2
+ \left (
 {\frac 12}  L_1 + 
\|\sigma_1\|^2_{
C([0,T], L^2(\R^n, l^2))
}
\right )
  \| v(t)\|^2_{l^2} .
 \ee

    By \eqref{exis_solp1}-\eqref{exis_solp5} we get
   for all $t\in (0,T)$,
    $$
    {\frac d{dt}}
   \| u(t) \|^2  
   + 2
   \| (-\Delta)^{\frac {\alpha}2} u(t)\|^2
   +  2 \lambda_1
   \| u(t) \|^p_{L^p(\R^n)}
   $$
  \be\label{exis_solp6}
     \le   
   \left ( 2
+ L_1 \| v(t) \|^2_{l^2}
\right )   \|u(t) \|^2
+ c_1
  \| v(t)\|^2_{l^2}
  + 
   \| g(t) \|^2 + 2 \| \psi_1 (t) \|_{L^1(\R^n)},
\ee
   where $c_1=  \left (  L_1 + 2
\|\sigma_1\|^2_{
C([0,T], L^2(\R^n, l^2))
}
\right  )$.
  By \eqref{exis_solp6} we find that
  for all $t\in [0, T]$,
  $$
   \| u(t) \|^2 
   \le 
    e^{   \int_0^t  
   \left (
   2
    +L_1    \| v(r)\|^2_{l^2}
   \right )
      dr }\|u_0\|^2
     $$
     $$
     + 
    \int_0^t
    e^{
        \int_s^t
   \left (
  2 
     +L_1    \| v(r)\|^2_{l^2}
   \right )
     dr 
      } 
      \left (
    c_1  \| v(s) \|^2_{l^2}
      +\| g(s) \|^2
      + 2 \| \psi_1 (s) \|_{L^1(\R^n)}
      \right )
      ds
   $$
   $$
     \le 
     e^{2T  + L_1  \int_0^T     \| v(r)\|^2_{l^2}
      dr }\|u_0\|^2
     $$
      \be\label{exis_solp7}
     +
      e^{2T  + L_1  \int_0^T     \| v(r)\|^2_{l^2}
      dr }
      \int_0^T
        \left (
    c_1  \| v(s) \|^2_{l^2}
      +\| g(s) \|^2
      + 2 \| \psi_1 (s) \|_{L^1(\R^n)}
      \right )
      ds.
     \ee

       By   \eqref{exis_solp7} we infer that
    for every  $R_1>0$  and $R_2>0$,
    there  exists
       $c_2=c_2(R_1,R_2, T)>0$ such that 
       for any $u_0\in H$ with
       $\| u_0 \|\le R_1$ and any
        $v\in L^2(0, T; l^2)$ with
        $\| v\|_{L^2(0, T; l^2)} \le R_2$,
        the solution $u$ satisfies, for all $t\in [0,T]$,
       \be\label{exis_solp8}
   \| u(t) \|^2 
   \le   c_2,\quad \forall \ t\in[0,T].
\ee

Integrating \eqref{exis_solp6}
on $[0,T]$, by \eqref{exis_solp8} we infer
that  exists
       $c_3=c_3(R_1,R_2, T)>0$ such that
         for all  $u_0\in H$ with
       $\| u_0 \|\le R_1$ and any
        $v\in L^2(0, T; l^2)$ with
        $\| v\|_{L^2(0, T; l^2)} \le R_2$, 
       \be\label{exis_solp9}
 2  \int_0^T \|(-\Delta)^{\frac {\alpha}2} u(t) \|^2
   dt 
   +2\lambda_1
   \int_0^T
   \| u(t)\|^p_{L^p(\R^n)} dt
   \le   c_3.
\ee 
Then  \eqref{exis_sol 2} follows from 
\eqref{exis_solp8}-\eqref{exis_solp9}.

   Next, we prove  \eqref{exis_sol 1}.
    Let $v_1, v_2$ be given in 
    $ L^2(0, T; l^2)$,
    and denote by
    $$
    u_1=u_{v_1}
    \quad \text{ and } \quad
    u_2=u_{v_2}.
    $$
    Suppose 
    $\| u_{0,1} \|\le R_1$,  $\| u_{0,2} \|\le R_1$, 
    $\| v_1\|_{L^2(0, T; l^2)}\le R_2$
    and $\| v_2\|_{L^2(0, T; l^2)}\le R_2$.
    Then by \eqref{exis_solp8} we  obtain
    \be
    \label{exis_solp11}
      \| u_1 (t) \| +   \| u_2 (t) \|
   \le   c_4,\quad \forall \ t\in[0,T],
    \ee
    where $c_4=c_4(R_1, R_2, T)>0$.
      By \eqref{contr1}-\eqref{contr2}
     we get
     $$
     {\frac d{dt}}
     \| u_1(t) -u_2(t)\|^2
     +2  \| (-\Delta)^{\frac {\alpha}2}
     (u_1(t)-u_2(t))\|^2
     $$
     $$
    =-
     2\int_{\R^n}
      (F(t,x, u_1(t)) -F(t,x, u_2(t)))  (u_1(t)-u_2(t)) dx
     $$
     \be\label{exis_solp12}
     +2
     \left (
     \sigma (t, u_1(t))v_1(t)-
     \sigma (t, u_2(t))v_2(t),
     \  u_1(t)-u_2(t)
     \right ).
     \ee
      By \eqref{F4} we have
  \be\label{exis_solp13}
   -
     2\int_{\R^n}
      (F(t,x, u_1(t)) -F(t,x, u_2(t)))  (u_1(t)-u_2(t)) dx
\le
2\|\psi_3(t) \|_{L^\infty(\R^n)} \| u_1(t)-u_2(t)\|^2.
\ee
  For the last term in \eqref{exis_solp12},
    we  have
   $$
   2
     \left (
     \sigma (t, u_1(t))v_1(t)-
     \sigma (t, u_2(t))v_2(t),
     \  u_1(t)-u_2(t)
     \right )
     $$ 
        $$
   \le 2
     \| 
     \sigma (t, u_1(t))v_1(t)-
     \sigma (t, u_2(t))v_2(t)\|
     \| u_1(t)-u_2(t) \|
     $$ 
     $$
   \le 2
     \| 
     (\sigma (t, u_1(t))-
     \sigma (t, u_2(t))
     v_1(t)\| 
     \| u_1(t)-u_2(t) \|
     $$ 
   $$
   + 2
     \|  
     \sigma (t, u_2(t))
     (v_1(t) - 
     v_2(t)) \|
     \| u_1(t)-u_2(t) \|
  $$
      $$
   \le 2
     \| 
    \sigma (t, u_1(t))-
     \sigma (t, u_2(t))\|_{\call_2( {l^2},H)}
     \|v_1(t)\| _{l^2}
     \| u_1(t)-u_2(t) \|
     $$ 
     \be\label{exis_solp15}
   + 2
     \|  
     \sigma (t, u_2(t))\|_{\call_2( {l^2},H)}
     \| v_1(t) - 
     v_2(t) \|_ {l^2}
     \| u_1(t)-u_2(t) \|.
     \ee
     By \eqref{sig7}   we see that
     \be\label{exis_solp16} 
      \| 
    \sigma (t, u_1(t))-
     \sigma (t, u_2(t))\|_{\call_2(l^2,H)}
     \le c_5    \| u_1(t)-u_2(t) \|,
     \quad \forall \  t\in [0,T].
     \ee 
    On the other hand,
       by \eqref{sig6} and \eqref{exis_solp11} we
       have
     \be\label{exis_solp17} 
      \|  
     \sigma (t, u_2(t))\|^2_{\call_2(l^2,H)}
     \le L_1(1+ \| u_2(t)  \|^2)
     +2 \|\sigma_1\|^2_{C([0,T], L^2(\R^n, l^2))}
     \le c_6^2 ,
     \quad \forall \  t\in [0,T] ,
     \ee 
     where $c_6=c_6(R_1,R_2,T)>0$.
     It follows from \eqref{exis_solp15}-\eqref{exis_solp17} 
     that for all $t\in [0,T]$,
      $$
   2
     \left (
     \sigma (t, u_1(t))v_1(t)-
     \sigma (t, u_2(t))v_2(t),
     \  u_1(t)-u_2(t)
     \right )
     $$ 
     $$
     \le 2 c_5 \| v_1(t)\|_{l^2}
     \| u_1(t)-u_2(t) \|^2
     +2c_6
     \| v_1(t)-v_2(t) \|_{l^2}
     \| u_1(t)-u_2(t) \|
     $$
         \be\label{exis_solp18} 
     \le   (c_6 + 2c_5  \| v_1(t)\|_{l^2}  )
     \| u_1(t)-u_2(t) \|^2
     +c_6  
     \| v_1(t)-v_2(t) \|_{l^2} ^2.
     \ee 
   
   By \eqref{exis_solp12}-\eqref{exis_solp13} 
   and  \eqref{exis_solp18}
    we obtain   for all $t\in (0, T)$
      $$
     {\frac d{dt}}
     \| u_1(t) -u_2(t)\|^2
     +2  \| (-\Delta)^{\frac {\alpha}2}
     (u_1(t)-u_2(t))\|^2
     \le   c_6  
     \| v_1(t)-v_2(t) \|_{l^2} ^2
     $$
         \be\label{exis_solp19}
   +
   (c_6 + 2\|\psi_3(t)\|_{L^\infty(\R^n)}
   + 2c_5  \| v_1(t)\|_{l^2}  )
     \| u_1(t)-u_2(t) \|^2 .
     \ee
     By \eqref{exis_solp19} we get 
     for all $t\in [0, T]$,
     $$
      \| u_1(t) -u_2(t)\|^2
      \le e^{
      \int_0^t 
      \left (
     c_6+ 
     2 \|\psi_3 (r) \|_{L^\infty(\R^n)} 
        +2 c_5  \| v_1(r)\|_{l^2}
        \right )
       dr
      } \| u_{0,1}  -u_{0,2}   \|^2
      $$
      $$
            +
          c_6  
    \int_0^t
    e^{
      \int_s^t 
       \left (
     c_6 + 
     2 \|\psi_3 (r) \|_{L^\infty(\R^n)}
        + 2c_5  \| v_1(r)\|_{l^2}
        \right )
      dr
      }
      \|v_1(s) - v_2 (s)  \|_{l^2}^2
      ds
      $$
    \be\label{exis_solp20}
      \le c_7 \| u_{0,1}  -u_{0,2}   \|^2
      +c_7
           \|v_1  - v_2    \|^2_{L^2(0,T; l^2)},
    \ee
    where $c_7=c_7(R_1,R_2, T)>0$.
    Integrating \eqref{exis_solp19}
    on $(0,T)$, by \eqref{exis_solp20} we get
       $$
    2 \int_0^T  \| (-\Delta)^{\frac {\alpha}2}
     (u_1(t)-u_2(t))\|^2  dt
      \le c_8 \| u_{0,1}  -u_{0,2}   \|^2
      +c_8
           \|v_1  - v_2    \|^2_{L^2(0,T; l^2)},
$$
      where $c_8=c_8(R_1,R_2, T)>0$,
      which along with \eqref{exis_solp20}
      yields 
    \eqref{exis_sol 1}.
 \end{proof}

 It follows from 
    Lemma \ref{exis_sol} that
 the solution
 $u_v$ of \eqref{contr1}-\eqref{contr2}
 is continuous in $C([0,T],H)
 \bigcap L^2(0,T; V)$ with respect to
  $v$ in the strong topology of  $L^2(0,T; l^2)$.
  However, for establishing the LDP of the stochastic
  equation, we need the 
    continuity of $u_v$  in
     $C([0,T],H)
 \bigcap L^2(0,T; V)$ with respect to
  $v$ in the weak topology of  $L^2(0,T; l^2)$
  rather than the strong topology.
  The following uniform tail-ends estimates
  of solutions are crucial for proving the 
     continuity of $u_v$  in $v$
     from
     the weak topology of  $L^2(0,T; l^2)$
    to the strong topology of  $C([0,T],H)
 \bigcap L^2(0,T; V)$.

 \begin{lem}\label{tail}
 Suppose   \eqref{F1}-\eqref{F4}
  and \eqref{sig0}-\eqref{sig3}  hold, 
  $T>0$ and $u_0\in H$.
  Then for every  
  $R>0$ and $\eps>0$,
  there exists $K=K(T, u_0, R, \eps)>0$ such that
  for all $v\in L^2(0,T; l^2)$
  with $\| v\|_{L^2(0,T; l^2)} \le R$,
  the solution
  $u_v$ of \eqref{contr1}-\eqref{contr2}
  satisfies, for all $m\ge K$  and  $t\in [0,T]$,
   $$
   \int_{|x|\ge m}
   \left |u_v(t, x) \right |^2 dx
   <\eps.
   $$
   \end{lem}
     
 \begin{proof}
 Let $\theta: \R^n \to [0,1]$  be a smooth function such that 
 \be\label{cutoff}
 \theta (x)=0 \ \ \text{for }  |x| \le {\frac 12};    \  \text{and} \  
 \theta (x) = 1 \  \  
 \text{for } \  |x| \ge 1.
\ee
Given $m\in \N$,   let 
 $\theta_m (x)  =  \theta  \left (
 {\frac xm}
 \right ) $.
  By \eqref{contr1}  we get
 $$
 {\frac {d}{dt}}
 \theta_m u_v(t)
  + \theta_m (-\Delta)^ \alpha  u_v  (t) 
   + \theta_m F(t, x, u_v) 
  = \theta_m g(t)
  +\theta_m
  \sigma (t, u_v(t)) v(t)  ,    
$$
and hence
 \be\label{tail 1}
 {\frac {d}{dt}}
 \|\theta_m u_v (t) \|^2 
 + 2 
   (   (-\Delta)^ {\frac {\alpha}2}   u_v(t),  
    (-\Delta)^ {\frac {\alpha}2} (\theta_m^2 u_v (t) ))
  + 2  \int_{\R^n}
   F(t, x, u_v (t))   \theta_m^2 (x)   u_v (t) dx  
 $$
 $$
 = 2   (\theta_m g (t), \theta_m  u_v (t) ) 
  + 2(  \theta_m
  \sigma (t, u_v(t)) v(t),
  \theta_m  u_v(t))  .
 \ee
 For the second term on the left-hand side of
 \eqref{tail 1} we have
  $$
 -2 
    (   (-\Delta)^ {\frac {\alpha}2}   u_v(t),  
     (-\Delta)^ {\frac {\alpha}2} (\theta_m^2 u_v(t)) )
   $$
   $$
   =-
   C(n,\alpha) 
   \int_{\R^n}\int_{\R^n} 
   {\frac
   { (u_v(t,x) -u_v(t,y)) (\theta_m^2(x) u_v(t,x)  
   -  \theta_m^2(y)  u_v(t,y)  ) }
   {|x-y|^{n+2\alpha}}
   }dxdy 
  $$ 
   $$
   \le  -
   C(n,\alpha) 
   \int_{\R^n}\int_{\R^n} 
   {\frac
   { (u_v(t,x) -u_v(t,y))
    (\theta_m^2(x)   -  \theta_m^2(y)  )  u_v(t,y)  }
   {|x-y|^{n+2\alpha}}
   }dxdy 
  $$
  $$
   \le
     C(n,\alpha)
    \int_{\R^n}
  | u_v(t,y)|  \left (
   \int_{\R^n} 
   {\frac
   { \left |
   (u_v(t,x) -u_v(t,y))
    (\theta_m(x)   -  \theta_m(y)  )(\theta_m(x)   +
     \theta_m(y)  ) \right |   }
   {|x-y|^{n+2\alpha}}
   }dx 
   \right )
   dy  
  $$
    $$
   \le
  2   C(n,\alpha)
  \| u_v (t) \|
  \left (  \int_{\R^n} \left (
   \int_{\R^n} 
   {\frac
   { |  (u_v (t,x) -u_v (t,y)) (\theta_m(x)   -  \theta_m(y)  )  | }
   {|x-y|^{n+2\alpha}}
   }dx
   \right )^2
   dy 
   \right )^{\frac 12} 
  $$
 $$
   \le
  2   C(n,\alpha)
   \| u_v (t) \|
   \left (
  \int_{\R^n}  \left (
   \int_{\R^n} 
   {\frac
   { |  u_v (t,x) -u_v (t,y)|^2  }
   {|x-y|^{n+2\alpha}}
   }dx
    \int_{\R^n} 
   {\frac
   { |  \theta_m(x) -\theta_m(y) |^2  }
   {|x-y|^{n+2\alpha}}
   }dx  \right )
   dy 
   \right )^{\frac 12}. 
$$
     $$
   \le
  2 c_1 C(n,\alpha)   m^{-\alpha}
     \| u_v(t) \|
   \left (  \int_{\R^n}  
   \int_{\R^n}
   {\frac
   { |  u_v(t,x) -u_v(t,y)|^2  }
   {|x-y|^{n+2\alpha}}
   }dx 
   dy \right )^{\frac 12} 
  $$
  $$
   \le
   c_1  C(n,\alpha) m^{-\alpha}
    \left ( 
    \| u_v(t) \|^2  
     + 
    \int_{\R^n}
 \int_{\R^n}
   {\frac
   { |  u_v(t,x) -u_v(t,y)|^2  }
   {|x-y|^{n+2\alpha}}
   }dx 
   dy   \right )
  $$
   \be\label{tail 2}
  =c_1  C(n,\alpha) m^{-\alpha}
   \| u_v (t) \|^2 
     +
     2  c_1   m^{-\alpha}    
     \| (-\Delta)^{\frac \alpha{2}} u_v (t) \|^2,
  \ee
  where $c_1$ is a positive constant
  independent of $m$.

  For the third term on the left-hand side of
  \eqref{tail 1}, by \eqref{F2} we get
  \be\label{tail 3}
  - 2 \int_{\R^n} 
  F(t, x, u_v (t))   \theta_m^2 (x)   u_v(t) dx  
  \le 2 \int_{\R^n}  \theta_m^2 (x) \psi_1(t,x) dx
  \le 2 \int_{|x| \ge {\frac 12 m}}  |\psi_1(t,x)| dx.
 \ee
By Young\rq{}s inequality, we have
  \be\label{tail 4}
 2  (\theta_m g(t), \theta_m u_v(t) )
 \le 
 \| \theta_m u_v (t)\|^2 
 +  \|  \theta _m   g(t) \|^2
 \le
  \| \theta_m u_v (t)\|^2 
 +   
\int_{|x| \ge {\frac 12} m} g^2(t,x) dx.
\ee
For the last term in \eqref{tail 1},
by \eqref{sig2}, \eqref{sig3} and \eqref{sig5}
we have
$$
2(  \theta_m
  \sigma (t, u_v(t)) v(t),
  \theta_m  u_v(t)) 
  \le
  2 \|  \theta_m
  \sigma (t, u_v(t)) v(t) \| \|
  \theta_m  u_v(t)\|
  $$
  $$
   \le
    \|  \theta_m
  \sigma (t, u_v(t)) v(t) \|^2
  +  \|
  \theta_m  u_v(t)\|^2
  \le
    \|  \theta_m
  \sigma (t, u_v(t))\|_{\call_2(l^2, H)}^2
  \| v(t) \|^2_{l^2}
  +  \|
  \theta_m  u_v(t)\|^2
  $$
   $$
  =   \| v(t) \|^2_{l^2} \sum_{k=1}^\infty
\| 
  \theta_m  
   \sigma_{1,k} (t)
  +  \theta_m   \kappa   \sigma_{2,k}
  (t,\cdot, u_v(t) )
\|^2 
  +  \|
  \theta_m  u_v(t)\|^2
  $$
   $$
  \le 2   \| v(t) \|^2_{l^2} \sum_{k=1}^\infty
( \| 
  \theta_m  
   \sigma_{1,k} (t)\|^2
  + \| \theta_m   \kappa   \sigma_{2,k}
  (t,\cdot, u_v(t) )
\|^2  )
  +  \|
  \theta_m  u_v(t)\|^2
  $$
   $$
  \le 2   \| v(t) \|^2_{l^2} \sum_{k=1}^\infty
  \| 
  \theta_m  
   \sigma_{1,k} (t)\|^2
    +  \|
  \theta_m  u_v(t)\|^2
  $$
  $$
  +
    2   \| v(t) \|^2_{l^2} \sum_{k=1}^\infty
  \int_{\R^n}
   \theta_m^2(x)    \kappa^2(x)  
   (\beta_k +\gamma_k |u_v(t)| )^2 dx
   $$
   $$
  \le 2   \| v(t) \|^2_{l^2} \sum_{k=1}^\infty
  \| 
  \theta_m  
   \sigma_{1,k} (t)\|^2
    +  \|
  \theta_m  u_v(t)\|^2
  $$
  $$
  +
   4   \| v(t) \|^2_{l^2} \sum_{k=1}^\infty
  \int_{\R^n}
   \theta_m^2(x)    \kappa^2(x)  
   (\beta_k^2 +\gamma_k^2 |u_v(t)|  ^2 )  dx
   $$
   $$
  \le 2   \| v(t) \|^2_{l^2} \sum_{k=1}^\infty
 \int_{|x| \ge {\frac 12} m}
  | \sigma_{1,k} (t,x ) |^2 dx
    +  \|
  \theta_m  u_v(t)\|^2
  $$
\be\label{tail 5}
  +
   4   \| v(t) \|^2_{l^2} \sum_{k=1}^\infty
   \beta_k^2 
  \int_{|x|\ge {\frac 12} m}     \kappa^2(x) dx
  +  4   \| v(t) \|^2_{l^2}\|\kappa\|_{L^\infty
  (\R^n)}^2
   \sum_{k=1}^\infty\gamma_k^2
 \| \theta_m u_v(t) \|^2.
 \ee

  It follows from
  \eqref{tail 1}-\eqref{tail 5} that
  for all $t\in (0,T)$,
  $$
 {\frac {d}{dt}}
 \|\theta_m u_v (t) \|^2 
 \le
 c_2 m^{-\alpha}
 \| u_v (t)\|^2_V
 +c_3
 (1+\| v(t)\|^2_{l^2})\|\theta_m u_v(t)\|^2
 + 2 \int_{|x| \ge {\frac 12} m}
    | \psi_1 (t,x)|
   dx
 $$
\be\label{tail 6}
 +  \int_{|x| \ge {\frac 12} m}
    | g (t,x)|^2
   dx
 +
 2   \| v(t) \|^2_{l^2} 
 \int_{|x| \ge {\frac 12} m}\sum_{k=1}^\infty
  | \sigma_{1,k} (t,x ) |^2 dx
    +  
    4   \| v(t) \|^2_{l^2} \sum_{k=1}^\infty
   \beta_k^2 
  \int_{|x|\ge {\frac 12} m}     \kappa^2(x) dx,
\ee
  where $c_2$ and $c_3$
  are positive numbers independent of $m$.
  
  Solving \eqref{tail 6} we obtain,
  for all $t\in [0,T]$,
  $$
   \|\theta_m u_v (t) \|^2 
   \le
   e^{c_3
   \int_0^t  (1+\| v(r)\|^2_{l^2})dr
   }\| \theta_m u_0\|^2
   +
 c_2 m^{-\alpha}  \int_0^t
  e^{c_3
   \int_s^t  (1+\| v(r)\|^2_{l^2})dr
   }
 \| u_v (s)\|^2_V ds
   $$
  $$
  +  
  \int_0^t
  e^{c_3
   \int_s^t  (1+\| v(r)\|^2_{l^2})dr
   }
   \left (
   \int_{|x| \ge {\frac 12} m}
   (2 | \psi_1 (s,x)| + |g(s,x)|^2)
   dx\right )  ds
 $$
 $$
 +
 2
  \int_0^t
  e^{c_3
   \int_s^t  (1+\| v(r)\|^2_{l^2})dr
   }
    \| v(s) \|^2_{l^2} \left (
 \int_{|x| \ge {\frac 12} m}\sum_{k=1}^\infty
  | \sigma_{1,k} (s,x ) |^2 dx
  \right )  ds
  $$
  $$
    +  
    4 \sum_{k=1}^\infty
   \beta_k^2 
   \left (
  \int_{|x|\ge {\frac 12} m}     \kappa^2(x) dx
  \right ) \int_0^t
  e^{c_3
   \int_s^t  (1+\| v(r)\|^2_{l^2})dr
   }
     \| v(s) \|^2_{l^2}  
  ds
  $$
   $$  
   \le
   e^{c_3   (T+  \| v \|^2_{L^2(0,T; l^2)}
     ) 
   }\| \theta_m u_0\|^2
   +
 c_2  e^{c_3   (T+  \| v \|^2_{L^2(0,T; l^2)}
     ) 
   }  m^{-\alpha}  \int_0^T 
 \| u_v (s)\|^2_V ds
   $$
  $$
  +    e^{c_3   (T+  \| v \|^2_{L^2(0,T; l^2)}
     ) 
   } 
   \int_{|x| \ge {\frac 12} m}
  \left (  \int_0^T
   (2 | \psi_1 (s,x)| + |g(s,x)|^2)
  ds \right )  dx
 $$
 $$
 +
 2 e^{c_3   (T+  \| v \|^2_{L^2(0,T; l^2)}
     ) 
   } 
  \int_0^T
    \| v(s) \|^2_{l^2} \left (
 \int_{|x| \ge {\frac 12} m}\sum_{k=1}^\infty
  | \sigma_{1,k} (s,x ) |^2 dx
  \right )  ds
  $$
  $$
    +  
    4 
     e^{c_3   (T+  \| v \|^2_{L^2(0,T; l^2)}
     ) 
   } \sum_{k=1}^\infty
   \beta_k^2 
   \left (
  \int_{|x|\ge {\frac 12} m}     \kappa^2(x) dx
  \right )
     \int_0^T
    \| v(s) \|^2_{l^2}  
  ds
  $$
   $$  
   \le
  c_4 \| \theta_m u_0\|^2
   +
 c_2 c_4   m^{-\alpha}  \int_0^T 
 \| u_v (s)\|^2_V ds
   $$
  $$
  +   c_4 
   \int_{|x| \ge {\frac 12} m}
  \left (  \int_0^T
   (2 | \psi_1 (s,x)| + |g(s,x)|^2)
  ds \right )  dx
 $$
 \be\label{tail 7}
 +
 2c_4
  \int_0^T
    \| v(s) \|^2_{l^2} \left (
 \int_{|x| \ge {\frac 12} m} 
  \| \sigma_{1 } (s,x ) \|^2_{l^2}  dx
  \right )  ds
    +  
    4 
    c_4 R^2
     \sum_{k=1}^\infty
   \beta_k^2 
   \left (
  \int_{|x|\ge {\frac 12} m}     \kappa^2(x) dx
  \right ),
\ee
 where
 $c_4=  e^{c_3   (T+  R^2) }
    $.
  
  Next, we deal with every term on the right-hand side of
  \eqref{tail 7}.
  Since $u_0\in H$,
  for  the first term
  on the right-hand side of \eqref{tail 7}  we have
   \be\label{7a}
   c_4 \| \theta_m u_0\|^2
   \le c_4
   \int_{|x|\ge {\frac 12} m}
   |u_0(x)|^2 dx
   \to 0, \ \text{ as } m \to \infty.
 \ee
   For  the 
   second  term
  on the right-hand side of \eqref{tail 7},
  by Lemma \ref{exis_sol} we have
 \be\label{tail 7b}
  c_2 c_4   m^{-\alpha}  \int_0^T 
 \| u_v (s)\|^2_V ds
 \le  c_2 c_4M_2   m^{-\alpha} \to 0 
 , \ \text{ as } m \to \infty.
 \ee
   Since $\psi_1\in L^1(0,T; L^1(\R^n))$
   and $g\in L^2(0,T; L^2(\R^n))$, 
   for the  
  third  term
  on the right-hand side of \eqref{tail 7},
 we have 
 \be\label{tail 7c}
 c_4 
   \int_{|x| \ge {\frac 12} m}
  \left (  \int_0^T
   (2 | \psi_1 (s,x)| + |g(s,x)|^2)
  ds \right )  dx
  \to 0 
 , \ \text{ as } m \to \infty.
  \ee
  
  By assumption, we know that
  $\sigma_1: [0,T]
  \to L^2(\R^n, l^2)$ is continuous, and hence
  the set
  $\{\sigma_1(s): s\in [0,T]\}$
  is a compact subset of 
  $L^2(\R^n, l^2)$, which implies that
  for every $\delta>0$,
  there exists $k\in \N$ and
  $s_1,\ldots, s_k\in [0,T]$ such that
  the open balls centered at
  $ \sigma_1 (s_i)$,
  $i=1,\ldots, k $, with radius 
  ${\frac 12}\delta$
  form a finite open cover of
  the set  $\{\sigma_1(s): s\in [0,T]\}$
  in
  $L^2(\R^n, l^2)$.
  Since  
   $ \sigma_1 (s_i)
   \in L^2(\R^n, l^2) $
   for each $i=1,\ldots,k$, we infer that
   there exists $K_1>0$ such that
   $$
   \left (
   \int_{|x|\ge K_1} \| \sigma_1 (s_i,x) \|^2_{l^2} dx
   \right )^{\frac 12}
   <{\frac 12} \delta,
   \quad \forall \ i=1,\ldots, k.
   $$
   Therefore,  we obtain
    $$
   \left (
   \int_{|x|\ge K_1} \| \sigma_1 (s ,x) \|^2_{l^2} dx
   \right )^{\frac 12}
   <  \delta,
   \quad \forall \ s\in [0,T],
   $$
 which implies that
 for all $m\ge 2 K_1$,
 $$
 2c_4
  \int_0^T
    \| v(s) \|^2_{l^2} \left (
 \int_{|x| \ge {\frac 12} m} 
  \| \sigma_{1 } (s,x ) \|^2_{l^2}  dx
  \right )  ds
  \le
   2c_4 \delta^2 
  \int_0^T
    \| v(s) \|^2_{l^2}  
    ds
    \le 2c_4 R^2 \delta^2 .
  $$
  Since $\delta>0$ is arbitrary, we infer from the
  above inequality that
  \be\label{tail 7d}
 2c_4
  \int_0^T
    \| v(s) \|^2_{l^2} \left (
 \int_{|x| \ge {\frac 12} m} 
  \| \sigma_{1 } (s,x ) \|^2_{l^2}  dx
  \right )  ds
  \to 0,
  \quad \text{as } \ m\to \infty.
  \ee
  Since $\kappa \in H$,
  for the  
 last  term
  on the right-hand side of \eqref{tail 7},
 we have 
\be\label{tail 7e} 
   4 
    c_4 R^2
     \sum_{k=1}^\infty
   \beta_k^2 
   \left (
  \int_{|x|\ge {\frac 12} m}     \kappa^2(x) dx
  \right )    \to 0,
  \quad \text{as } \ m\to \infty.
  \ee
  
  It follows from \eqref{tail 7}-\eqref{tail 7e}
  that 
  for every $\eps>0$,
  there exists $K_2=K_2(T, u_0, R, \eps)>0$
  such that  for all  $m\ge K_2$ and $t\in [0,T]$,
  $$
  \int_{|x|\ge m} |u_v (t,x) |^2 dx
  \le 
   \|\theta_m u_v (t) \|^2 <\eps,
   $$
  which concludes the proof.
 \end{proof}
 
 The following lemma is concerned with
  the 
    continuity of
    the solution 
    $u_v$ of problem  \eqref{contr1}-\eqref{contr2}
     in
     $C([0,T],H)
 \bigcap L^2(0,T; V)$ with respect to
  $v$ in the weak topology of  $L^2(0,T; l^2)$.

 \begin{lem}\label{wc_sol}
 Suppose   \eqref{F1}-\eqref{F4}
  and \eqref{sig0}-\eqref{sig3}  hold.
 Let
  $v, v_n \in L^2(0,T; l^2)$ 
  for all  $n\in \N$
  and  $u_v$,  $u_{v_n}$
  be the solutions of 
  \eqref{contr1}-\eqref{contr2}
  corresponding to $v$  and $v_n$,
  respectively.
    If   $v_n \to v$ weakly in 
  $L^2(0,T; l^2)$,
  then
  $ u_{v_n} \to u_v $ strongly
  in  
 $  C([0,T], H)\bigcap L^2(0,T;V)
 $. 
   \end{lem}
 
 \begin{proof}
 The proof consists of four steps.
  Suppose $v_n \to v$ weakly in 
  $L^2(0,T; l^2)$.  We first derive the uniform estimates
  of solutions.
  
  {\bf Step 1: uniform estimates of solutions.}
 Since $v_n \to v$ weakly in 
  $L^2(0,T; l^2)$,  we know
  $\{v_n\}_{n=1}^\infty$ is bounded in
  $L^2(0,T; l^2)$, and thus by
   Lemma  \ref{exis_sol}
  we see that
  there exists $c_1=c_1(T)>0$ such that
  for all $n\in \N$,
  \be\label{wc_sol 1}
  \| u_{v_n}\|_{C([0,T], H)}
  + \| u_{v_n}\|_{L^2(0,T; V)  }
  + \| u_{v_n}\|_{L^p(0,T; L^p(\R^n) )  }
  + \| v_n \|_{L^2(0,T;l^2)  }
  \le c_1,
  \ee
  which along with \eqref{F5} shows that
  \be\label{wc_sol 2}
   \| F(\cdot, \cdot, u_{v_n} )\| 
   _{L^q(0,T; L^q(\R^n))}
  \le c_2,
  \ee
   where $c_2=c_2(T)>0$.

  Similar to \eqref{defnsol_3}, by \eqref{contr1}-\eqref{contr2}
  we have, for all $0 \le t\le T$,
   \be\label{wc_sol 2a}
  u_{v_n} (t)
  +\int_0^t (-\Delta)^\alpha u_{v_n}(s) ds
  +\int_0^t F(s,\cdot, u_{v_n}(s)) ds
  =u_0 +\int_0^t g(s) ds
  +\int_0^t \sigma (s, u_{v_n}(s)) v_n(s) ds,
\ee
 in $\left (V\bigcap L^p(\R^n)\right )^*$, which implies that
 for all $0\le  s \le t \le T$,
  $$
 \| u_{v_n} (t)- u_{v_n} (s)\|
 _{\left (V\bigcap L^p(\R^n)\right )^*}
\le \int_s^t
\|  (-\Delta)^\alpha u_{v_n}(s) 
\|_{V^*}ds
$$
  \be\label{wc_sol 3}
  +\int_s^t \|  F(s,\cdot, u_{v_n}(s)) \|_{L^q(\R^n)}ds
  +   \int_s^t \| g(s) \| ds
  +\int_s^t \| \sigma (s, u_{v_n}(s)) v_n(s)\| ds.
\ee
By \eqref{wc_sol 1} we have
   \be\label{wc_sol 4}
 \int_s^t
\|  (-\Delta)^\alpha u_{v_n}(s) 
\|_{V^*}ds
\le  \int_s^t
\|   u_{v_n}(s) 
\|_V  ds
\le (t-s)^{\frac 12}
\|   u_{v_n} 
\|_{L^2(0,T; V)}
\le c_1 (t-s)^{\frac 12}.
\ee
By \eqref{wc_sol 2} we have
   \be\label{wc_sol 5}
 \int_s^t \|  F(s,\cdot, u_{v_n}(s)) \|_{L^q(\R^n)}ds
 \le (t-s)^{\frac 1p}
 \| F(\cdot, \cdot, u_{v_n})\|_{L^q(0,T; L^q(\R^n))}
 \le  c_2 (t-s)^{\frac 1p}.
 \ee
 By Holder\rq{}s inequality we have
  \be\label{wc_sol 6}
  \int_s^t \| g(s)\| ds
  \le (t-s)^{\frac 12} \| g \|_{L^2(0,T; H)}.
  \ee
 For the last term in \eqref{wc_sol 3},
  by \eqref{sig6} and \eqref{wc_sol 1}
  we have
$$
 \int_s^t \| \sigma (s, u_{v_n}(s)) v_n(s)\| ds
 \le
   \int_s^t \| \sigma (s, u_{v_n}(s))\|_{\call_2(l^2,H)}
\|  v_n(s)\|_{l^2}  ds
 $$
 $$
  \le
  c_3 \int_s^t 
   \left ( 1
    +  \| u_{v_n} (s) \|
   + \|\sigma_1\|_{C([0,T],
   L^2(\R^n, l^2) )}
   \right )
\|  v_n(s)\|_{l^2}  ds
$$
$$
  \le
  c_3  \left ( 1
    +  c_1
   + \|\sigma_1\|_{C([0,T],
   L^2(\R^n, l^2) )}
   \right )
   \int_s^t 
 \|  v_n(s)\|_{l^2}  ds
$$
$$
  \le
  c_3  \left ( 1
    +  c_1
   + \|\sigma_1\|_{C([0,T],
   L^2(\R^n, l^2) )}
   \right )
 (t-s)^{\frac 12}
 \|  v_n \|_{L^2(0,T; l^2)}   
$$
   \be\label{wc_sol 7}
  \le
  c_1 c_3  \left ( 1
    +  c_1
   + \|\sigma_1\|_{C([0,T],
   L^2(\R^n, l^2) )}
   \right )
 (t-s)^{\frac 12} .
 \ee
 
 Since $p>2$, by   \eqref{wc_sol 3}-\eqref{wc_sol 7}
  we obtain, 
  for all $0\le  s \le t \le T$,
  \be\label{wc_sol 8}
 \| u_{v_n} (t)- u_{v_n} (s)\|
 _{\left (V\bigcap L^p(\R^n)\right )^*}
\le c_4 (t-s)^{\frac 1p},
\ee
where $c_4=c_4(T)>0$.
 
  By \eqref{wc_sol 1}
  and \eqref{wc_sol 2}
   we infer that
  there exists
  $z\in 
  L^\infty(0,T; H)\bigcap
  {L^2(0,T; V)  }\bigcap
  {L^p(0,T; L^p(\R^n) )  }$,
  $\varphi \in L^q(0,T; L^q(\R^n))$
  and  a subsequence of
  $\{u_{v_n}\}_{n=1}^\infty$
  (not relabeled) such that
  \be\label{wc_sol 10}
  u_{v_n}
  \to z
  \ \text{ weak-star  in  }  \  L^\infty(0,T; H),
  \ee
   \be\label{wc_sol 11}
  u_{v_n}
  \to z
  \ \text{ weakly  in  }  \  L^2(0,T; V),
  \ee 
   \be\label{wc_sol 12}
  u_{v_n}
  \to z
  \ \text{ weakly  in  }  \   L^p(0,T; L^p(\R^n) ) ,
  \ee
  and
  \be\label{wc_sol 13}
  F(\cdot,\cdot, u_{v_n})
  \to \varphi
  \ \text{ weakly  in  }  \   L^q(0,T; L^q(\R^n) ) .
  \ee

  We will prove $z=u_v$
  which is the unique solution
  of \eqref{contr1}-\eqref{contr2},
  for which a strong convergence of $\{u_{v_n}\}_{n=1}^\infty$
  is needed.
  
  {\bf Step 2: prove the   convergence:}
   \be\label{wc_sol 14}
     u_{v_n}
  \to z
  \ \text{ strongly  in  }  \  C([0,T], 
   (V\bigcap L^p(\R^n)   )^*).
  \ee 
  The Arzela-Ascoli theorem
  will be
   employed to show
  \eqref{wc_sol 14}, for which we need to
  verify:
  
  (i)  For every $t\in [0,T]$, the set
  $\{ u_{v_n} (t)\}_{n=1}^\infty$ is precompact
  in $ 
    (V\bigcap L^p(\R^n)   )^* $.
   
   (ii) The sequence $\{u_{v_n}\}_{n=1}^\infty$
   is equicontinuous  in 
   $ 
   (V\bigcap L^p(\R^n)   )^* $ on $[0,T]$.
   
  Note that the equicontinuity of 
   $\{u_{v_n}\}_{n=1}^\infty$
   in 
   $ 
   (V\bigcap L^p(\R^n)   )^* $
   follows from \eqref{wc_sol 8} immediately.
   So we only need to show
     (i).
     
     Let  $\theta$ be the cutoff function as given by
     \eqref{cutoff} and $\theta_m (x)
     =\theta ({\frac xm})$
     for $m\in \N$  and $x\in \R^n$.
     Given  $t\in [0,T]$, define
        \be\label{wc_sol 18}
        {\widetilde{u}}^m_{v_n} (t,x)
        =\theta_m (x)   u _{v_n} (t,x) 
        \quad \text{and}
        \quad
          {\widehat{u}}^m_{v_n} (t,x)
        =(1-\theta_m (x) )  u _{v_n} (t,x),
        \quad \forall \ x\in \R^n.
        \ee
        Then we have
         $ u _{v_n} (t)
         = {\widetilde{u}}^m_{v_n} (t)
         +  {\widehat{u}}^m_{v_n} (t)$
         for  $t\in [0,T]$ and $m\in \N$.
         By \eqref{wc_sol 1}  and
         \eqref{wc_sol 18} we see that
           \be\label{wc_sol 19}
          \|   {\widehat{u}}^m_{v_n }
          \|_{C([0,T], H)}
          \le   \|   u_{v_n }
          \|_{C([0,T], H)}
          \le c_1, \quad \forall \  n,m\in \N.
          \ee
          Given $m\in \N$, denote by
          $$
          Z_m =   \{u\in H:  \| u\| \le c_1,
          \  u(x) =0  \
          \text{for almost  all }  |x|\ge m\}.
          $$
         Then by
          \eqref{wc_sol 18}
          and \eqref{wc_sol 19},
          we know 
            that
             $\{{\widehat{u}}^m_{v_n }(t)\}_{n=1}^\infty
             \subseteq 
           Z_m$
          for every $m\in \N$  and $t\in [0,T]$.
          Since  the embedding $H^\alpha (|x|<m)
          \hookrightarrow L^2(|x|<m)$ is compact, we
          infer 
           that the set
          $Z_m$ is a precompact  subset
          of   $ 
   (V\bigcap L^p(\R^n)   )^* $
   for every $m\in \N$. Therefore
    for every $m\in \N$  and $t\in [0,T]$,
    \be\label{wc_sol 20}
       \{{\widehat{u}}^m_{v_n }(t)\}_{n=1}^\infty
         \ \text{ is precompact in }
         \  (V\bigcap L^p(\R^n)   )^*.
    \ee

          On the other hand, by
          \eqref{wc_sol 1}
         and   Lemma \ref{tail}
           we find that
          for every $\eps>0$, there exists
          $m_0=m_0(T, u_0, \eps)>0$ such that
          for all $m\ge m_0$ and  $t\in [0,T]$,
          $$
          \int_{|x|\ge {\frac 12} m}
          |u_{v_n} (t,x)|^2 dx <{\frac 14} \eps^2,
          \quad \forall \  n\in \N,
          $$
          which implies that for all
          $m\ge m_0$, 
        $ n\in \N$ and  $t\in [0,T]$,
          \be\label{wc_sol 21}
       \| {\widetilde{u}}^m_{v_n} (t )\|^2
       =\int_{|x|\ge {\frac 12} m}
       \theta_m^2 (x) |u_{v_n} (t,x)|^2 dx
       \le
       \int_{|x|\ge {\frac 12} m}
        |u_{v_n} (t,x)|^2 dx<{\frac 14} \eps^2.
   \ee

       By \eqref{wc_sol 20}-\eqref{wc_sol 21} we 
       find that  
         $  
          \{ {\widetilde{u}}^{m_0}_{v_n} (t)
         +  {\widehat{u}}^{m_0}_{v_n} (t)\}
         _{n=1}^\infty$
         has a finite open cover with
         radius $\eps$ in  
         $(V\bigcap L^p(\R^n)   )^*$.
         Since 
          $ u _{v_n} (t)
         = {\widetilde{u}}^m_{v_n} (t)
         +  {\widehat{u}}^m_{v_n} (t)$
         for   all $m\in \N$, we infer that
         for every $\eps>0$,
         the sequence
        $ \{u _{v_n} (t)\}_{n=1}^\infty$
         has a finite open cover with
         radius $\eps$ in  
         $(V\bigcap L^p(\R^n)   )^*$,
         which proves (ii).

         Then by \eqref{wc_sol 10}
         and the Arzela-Ascoli theorem, we find
         that there exists a further subsequence
         of  $\{u _{v_n}  \}_{n=1}^\infty$
         (not relabeled)  which satisfies 
         \eqref{wc_sol 14}.

  Next,
   we  show
   the strong convergence of
  $\{u _{v_n}  \}_{n=1}^\infty$
  in $L^2(0,T; H)$.

  {\bf Step 3: prove the   convergence:}
          \be\label{wc_sol 21a}
         u_{v_n}      \to  u_v
         \ \text{ strongly in } \
         L^2(0,T; H).
         \ee

         Note that
         $$
         \int_0^T
         \| u_{v_n} (t)  -z (t)\|^2 dt
         =\int_0^T
         (u_{v_n} (t)  -z (t),\
         u_{v_n} (t)  -z (t))_{
          (V\bigcap L^p(\R^n),
          (V\bigcap L^p(\R^n)   )^* ) 
         }
         $$
         $$
         \le
         \left (
         \int_0^T 
         \| u_{v_n} (t)  -z (t)\|_{
        V\bigcap L^p(\R^n)  }^2 dt
        \right )^{\frac 12}
         \left (
         \int_0^T 
         \| u_{v_n} (t)  -z (t)\|_{
         (V\bigcap L^p(\R^n)   )^*  }^2 dt
        \right )^{\frac 12} ,
        $$ 
          which along with
          \eqref{wc_sol 1}
         and \eqref{wc_sol 14} shows that
        \be\label{wc_sol 22}
         u_{v_n}      \to z 
         \ \text{ strongly in } \
         L^2(0,T; H).
         \ee    
         
         To prove \eqref{wc_sol 21a},
         by \eqref{wc_sol 22}  we only need to show
         $z=u_v$.
         By \eqref{wc_sol 2a} we have, 
       for all $0 \le t\le T$
       and $\xi \in  V\bigcap L^p(\R^n) $,
   \be\label{wc_sol 22a}
  (u_{v_n} (t), \xi)
  +\int_0^t 
  ( (-\Delta)^{\frac {\alpha}2}  u_{v_n}(s),
  (-\Delta)^{\frac {\alpha}2}  \xi
  ) ds
  +\int_0^t \int_{\R^n} F(s,x, u_{v_n}(s))
  \xi (x) dx  ds
  $$
  $$
  =(u_0,\xi)  +\int_0^t (g(s),\xi)  ds
  +\int_0^t 
  (\sigma (s, u_{v_n}(s)) v_n(s), \xi)  ds.
\ee 
We  now prove $z=u_v$ by taking the limit
of
\eqref{wc_sol 22a} as $n\to \infty$.
 We first deal with the nonlinear  term 
on the left-hand side of \eqref{wc_sol 22a}.
       By \eqref{wc_sol 22} we infer that,
       up to a subsequence,
           \be\label{wc_sol 23}
         u_{v_n}      \to z 
         \ \text{ a.e. on  } \
         (0,T)\times \R^n.
         \ee
         By the continuity of $F$ and \eqref{wc_sol 23}
         we obtain
         $$
         F(t,x, u_{v_n})
         \to F(t,x, z)
             \ \text{ a.e. on  } \
         (0,T)\times \R^n
         $$
         which together
         with \eqref{wc_sol 13}
         and  Mazur\rq{}s  theorem  implies
         that
         $\varphi = F(t,x,z)$ and hence
         \be\label{wc_sol 24}
         F(t,x, u_{v_n})
         \to F(t,x, z)
             \ \text{ weakly in  } \
         L^q(0,T; L^q(\R^n)).
       \ee
       
       Next, 
       We   deal with the nonlinear  term 
on the right-hand side of \eqref{wc_sol 22a}
by showing  that for every
       $t\in [0,T]$,
       \be\label{wc_sol 25}
      \lim_{n\to \infty}
       \int_0^t
       \left ( \sigma (s, u_{v_n}(s))-
        \sigma (s, z (s)) \right )
         v_n(s)   ds  =0
         \ \text{ in } \ H.
         \ee
    By \eqref{sig7}
    and \eqref{wc_sol 1} we obtain
     $$
     \|  \int_0^t
       \left ( \sigma (s, u_{v_n}(s))-
        \sigma (s, z (s)) \right )
         v_n(s)   ds\|
         \le
          \int_0^t
       \| \sigma (s, u_{v_n}(s))-
        \sigma (s, z (s))\|_{\call_2(l^2, H)}
        \|
         v_n(s)\|_{l^2}   ds
         $$
         $$
       \le \| \kappa\|_{L^\infty (\R^n)}
      \left ( \sum_{k=1}^\infty
       \alpha_k^2
       \right )^{\frac 12}
          \int_0^t
       \|  u_{v_n}(s) 
         -z (s)  \|
        \|
         v_n(s)\|_{l^2}   ds
         $$
      $$
       \le \| \kappa\|_{L^\infty (\R^n)}
      \left ( \sum_{k=1}^\infty
       \alpha_k^2
       \right )^{\frac 12}
         \left (
          \int_0^T
       \|  u_{v_n}(s) 
         -z (s)  \|^2 ds
         \right )^{\frac 12}
       \left (
          \int_0^T  \|
         v_n(s)\|_{l^2}^2   ds
         \right )^{\frac 12}
         $$
       $$
       \le c_1 \| \kappa\|_{L^\infty (\R^n)}
      \left ( \sum_{k=1}^\infty
       \alpha_k^2
       \right )^{\frac 12}
         \left (
          \int_0^T
       \|  u_{v_n}(s) 
         -z (s)  \|^2 ds
         \right )^{\frac 12},
      $$
      which along with \eqref{wc_sol 22}
      gives \eqref{wc_sol 25}.

     We now  prove that
       for every
       $t\in [0,T]$,
              \be\label{wc_sol 26}
      \lim_{n\to \infty}
       \int_0^t
         \sigma (s, z (s))  
         v_n(s)   ds  =  \int_0^t
         \sigma (s, z (s))  
         v (s)   ds
         \ \text{ weakly  in } \ H.
         \ee
         Given $t\in [0,T]$, define an operator
         $\calg: L^2(0,T; l^2)
         \to H$ by
         $$
         \calg (\widetilde{v})
         =\int_0^t  \sigma (s, z (s))  
         {\widetilde{v}} (s)   ds,
         \quad \forall \  \widetilde{v}
         \in L^2(0,T; l^2).
         $$
         By \eqref{sig6} and \eqref{wc_sol 10}, we
         know that
         $\calg$ is well-defined, and
          $$
          \| \calg (\widetilde{v})\|
          \le
          \int_0^t \| \sigma (s, z (s))\|
          _{\call_2(l^2,H)}  
         \| {\widetilde{v}} (s) \|_{l^2}
           ds
           $$
           $$ 
          \le
         \left ( \int_0^T \| \sigma (s, z (s))\|^2
         _{\call_2(l^2,H)}  ds
         \right )^{\frac 12}
            \left ( \int_0^T
         \| {\widetilde{v}} (s) \|_{l^2}^2 ds
         \right )^{\frac 12}
           $$
      $$ 
          \le
         \left ( \int_0^T
         \left (L_1 +L_1 \| z(s)\|^2
         +2 \| \sigma_1 (s) \|^2_{L^2(\R^n,
          l^2)} \right )
          ds
         \right )^{\frac 12}
            \left ( \int_0^T
         \| {\widetilde{v}} (s) \|_{l^2}^2 ds
         \right )^{\frac 12}
           $$
            $$ 
          \le T
         \left ( L_1  
          +L_1   \| z \|^2_{L^\infty(0,T; H)}
         +2  \| \sigma_1 \|^2_{C([0,T], L^2(\R^n,
          l^2) )}  
        \right )^{\frac 12}
          \| {\widetilde{v}}  \|_{L^2(0,T; l^2 )}  ,
           $$
         which shows that
         $\calg: L^2(0,T; l^2)
         \to H$ is a linear bounded operator,
         and hence it is weakly continuous.
         Since 
         $v_n \to v$ weakly in 
         $L^2(0,T; l^2)$, we obtain
         $\calg (v_n) \to \calg(v)$ weakly
         in $H$, which gives \eqref{wc_sol 26}.
         
        It follows from
        \eqref{wc_sol 25}-\eqref{wc_sol 26} that
        for all $\xi \in V\bigcap L^p(\R^n)$,
     $$    \int_0^t 
  (\sigma (s, u_{v_n}(s)) v_n(s), \xi)  ds
  $$
  $$
  =
 \left (\xi, \  \int_0^t 
   (\sigma (s, u_{v_n}(s))
  -
  \sigma (s, z (s) ))
   v_n(s)   ds
   \right )
   +  \left (\xi, \  \int_0^t  
  \sigma (s, z (s)) 
   v_n(s)  ds \right )
      $$   
   \be\label{wc_sol 27}
   \to     \int_0^t  
  (\sigma (s, z (s)) 
   v(s), \ \xi)   ds,
   \quad \text{as} \ n \to \infty.
   \ee     
    Letting $n\to \infty$ in \eqref{wc_sol 22a},
    by
    \eqref{wc_sol 11},
    \eqref{wc_sol 14} and
         \eqref{wc_sol 27}, we obtain that
         for all $0 \le t\le T$
       and $\xi \in  V\bigcap L^p(\R^n) $,
  $$
  (z (t), \xi)
  +\int_0^t 
  ( (-\Delta)^{\frac {\alpha}2}  z(s),
  (-\Delta)^{\frac {\alpha}2}  \xi
  ) ds
  +\int_0^t \int_{\R^n} F(s,x, z(s))
  \xi (x) dx  ds
  $$
  $$
  =(u_0,\xi)  +\int_0^t (g(s),\xi)  ds
  +\int_0^t 
  (\sigma (s, z(s)) v (s), \xi)  ds,
$$  
which 
shows that $z$ is a solution of \eqref{contr1}-\eqref{contr2}.
Since the solution of 
\eqref{contr1}-\eqref{contr2} is unique,
we have $z=u_v$, which together with
\eqref{wc_sol 22} yields \eqref{wc_sol 21a}.
Next, we improve the convergence of 
 \eqref{wc_sol 21a}.

  {\bf Step 4: prove the convergence:}
          \be\label{wc_sol 30}
         u_{v_n}      \to  u_v
         \ \text{ strongly in } \
         C([0,T], H) \bigcap L^2(0,T; V).
         \ee
    By \eqref{contr1}-\eqref{contr2}
     we  have
     $$
     {\frac d{dt}}
     \|  u_{v_n}   (t) - u_v  (t)\|^2
     +2  \| (-\Delta)^{\frac {\alpha}2}
     ( u_{v_n}   (t) - u_v  (t) )\|^2
     $$
     $$
    =-
     2\int_{\R^n}
      (F(t,x,  u_{v_n} (t)) -F(t,x,  u_v(t))) 
       ( u_{v_n}   (t) - u_v  (t) ) dx
     $$
     \be\label{wc_sol 31}
     +2
     \left (
     \sigma (t,  u_{v_n} (t))v_n(t)-
     \sigma (t,  u_v(t))v (t),
     \   u_{v_n}   (t) - u_v  (t)
     \right ).
     \ee
      By \eqref{F4} we have
   \be\label{wc_sol 32}
    -
     2\int_{\R^n}
      (F(t,x,  u_{v_n} (t)) -F(t,x,  u_v(t))) 
       ( u_{v_n}   (t) - u_v  (t) ) dx
\le
2\|\psi_3(t) \|_{L^\infty(\R^n)} \|u_{v_n}   (t) - u_v  (t) \|^2.
\ee
  For the last term in \eqref{wc_sol 31},
       we have  for all $t\in [0,T]$,
      $$
      2
     \left (
     \sigma (t,  u_{v_n} (t))v_n(t)-
     \sigma (t,  u_v(t))v (t),
     \   u_{v_n}   (t) - u_v  (t)
     \right )
       $$  
      \be\label{wc_sol 33}
      \le 2
     \left (
    \| \sigma (t,  u_{v_n} (t))\|_{\call_2(l^2,H)}
    \|v_n(t)\|_{l^2}
    + \|
     \sigma (t,  u_v(t))\|_{\call_2(l^2, H)}
     \| v (t)\|_{l^2} \right )
    \| u_{v_n}   (t) - u_v  (t)
  \|.
  \ee 
       By  \eqref{sig6} and \eqref{wc_sol 1}, 
         we have  for all $t\in [0,T]$,
      \be\label{wc_sol 34}
    \| \sigma (t,  u_{v_n} (t))\|_{\call_2(l^2,H)}^2
    \le
    L_1 (1+ \|  u_{v_n} (t)\|^2)
    +2 \|\sigma_1\|^2_{
    C([0,T], L^2(\R^n, l^2))
    }
    \le c_5,
    \ee
    where $c_5=c_5(T)>0$.
    Similarly,    we have  for all $t\in [0,T]$,
     \be\label{wc_sol 35}
    \| \sigma (t,  u_v(t))\|_{\call_2(l^2,H)}^2
    \le
    L_1 (1+ \|  u_v(t)\|^2)
    +2 \|\sigma_1\|^2_{
    C([0,T], L^2(\R^n, l^2))
    }
    \le c_6,
    \ee
    where $c_6=c_6(T)>0$.
    It follows  from
    \eqref{wc_sol 33}-\eqref{wc_sol 35}
    that   for all $t\in [0,T]$,
      $$
      2
     \left (
     \sigma (t,  u_{v_n} (t))v_n(t)-
     \sigma (t,  u_v(t))v (t),
     \   u_{v_n}   (t) - u_v  (t)
     \right )
     $$
    \be\label{wc_sol 36}
    \le c_7
     \left (\| v_n(t)\|_{l^2}
     +  \|v (t) \|_{l^2} \right ) 
     \|   u_{v_n}   (t) - u_v  (t)\|,
   \ee
      where $c_7=c_7(T)>0$.
    By \eqref{wc_sol 31}, \eqref{wc_sol 32} 
    and
    \eqref{wc_sol 36} 
    we obtain   for all $t\in (0, T)$,
      $$
     {\frac d{dt}}
     \|  u_{v_n}   (t) - u_v  (t)\|^2
     +2  \| (-\Delta)^{\frac {\alpha}2}
     ( u_{v_n}   (t) - u_v  (t) )\|^2
   $$  
      \be\label{wc_sol 37}
    \le    
2\|\psi_3(t) \|_{L^\infty(\R^n)}
 \|u_{v_n}   (t) - u_v  (t) \|^2
     +
   c_7
     \left (\| v_n(t)\|_{l^2}
     +  \|v (t) \|_{l^2} \right ) 
     \|   u_{v_n}   (t) - u_v  (t)\|.
   \ee
   Integrate \eqref{wc_sol 37}
   on $(0,t)$ with $t\in [0,T]$ to obtain
       $$ 
     \|  u_{v_n}   (t) - u_v  (t)\|^2
     +2  
     \int_0^t
     \| (-\Delta)^{\frac {\alpha}2}
     ( u_{v_n}   (s) - u_v  (s) )\|^2 ds
   $$  
     $$
    \le    
2\int_0^t \|\psi_3(s) \|_{L^\infty(\R^n)}
 \|u_{v_n}   (s) - u_v  (s) \|^2 ds
     +
   c_7\int_0^t
     \left (\| v_n(s)\|_{l^2}
     +  \|v (s) \|_{l^2} \right ) 
     \|   u_{v_n}   (s) - u_v  (s)\| ds
$$
  \be\label{wc_sol 38}
    \le    
2\|\psi_3 \|_{L^\infty(0,T;
L^\infty(\R^n) )} 
 \|u_{v_n}    - u_v   \|^2 _{L^2(0,T; H)} 
     +
   c_7 
     \left (\| v_n \|_{L^2(0,T; l^2 )}
     +  \|v  \| _{L^2(0,T; l^2 )} \right ) 
     \|u_{v_n}    - u_v   \|  _{L^2(0,T; H)}.
\ee
By \eqref{wc_sol 1}
and \eqref{wc_sol 21a} we see  that
the right-hand side of 
  \eqref{wc_sol 38}
  converges  to zero as $n\to \infty$,
  from which \eqref{wc_sol 30} follows.
   \end{proof}

  Let $\calg^0: C([0,T], U) \to C([0,T], H)
  \bigcap L^2(0,T; V)$
  be a map given 
by, for every $\xi \in  C([0,T], U)$, 
  \be\label{calg0}
\calg^0 (\xi)
=
\left \{
\begin{array}{ll}
u_v & \text{ if } \xi= \int_0^\cdot v(t) dt
\ \text{ for some } v\in L^2(0,T; l^2);\\
0, &  \text{ otherwise} ,
\end{array}
\right.
\ee
 where $u_v$ is the
 solution
 of \eqref{contr1}-\eqref{contr2}.
  Given  
 $\phi \in   C([0,T], H)
  \bigcap L^2(0,T; V)$, denote by
\be\label{rate_LS}
 I(\phi)
 =\inf
 \left \{
 {\frac 12} \int_0^T \| v(s)\|_{l^2} ^2 ds:
 \ v\in L^2(0,T; l^2), \ u_v =\phi
 \right \},
\ee
 where $u_v$ is the
 solution
 of \eqref{contr1}-\eqref{contr2}.
Again, by default,   $\inf \emptyset =\infty$.

 Recall that for every $N>0$,
$ 
S_N =\{ v\in L^2(0,T; l^2): \int_0^T
\| v(t) \|^2_{l^2} dt \le N\}.
$ 
  We will prove 
 the solutions of \eqref{intr1}-\eqref{intr3}
 satisfies the LDP  
 in  $C([0,T], H)
  \bigcap L^2(0,T; V)$ as $\eps \to 0$,
  for which the following lemma is needed.

 \begin{lem}\label{clev}
 Suppose   \eqref{F1}-\eqref{F4}
  and \eqref{sig0}-\eqref{sig3}  hold.
Then for every $N<\infty$, the set
\be\label{clev 1}
K_N
=\left \{
\calg^0
\left (
\int_0^\cdot v(t) dt
\right ) : \  v\in S_N
\right \}
\ee
is a compact subset
of $ C([0,T], H)
  \bigcap L^2(0,T; V)$.
 \end{lem}

 \begin{proof}
 By \eqref{calg0} we see that
 $  K_N
=\left \{
u_v : \  v\in S_N
\right \} 
$, 
where $u_v$
is the solution of \eqref{contr1}-\eqref{contr2}.
Suppose 
  $\{u_{v_n}\}_{n=1}^\infty$
 is a sequence
in $K_N$ with $v_n \in S_N$
for all $n\in \N$.
 Then  
there exists $v\in S_N$ and a subsequence
$\{ {v_{n_k}}\}_{k=1}^\infty$
such that
$v_{n_k} \to v$ weakly in 
$L^2(0,T; l^2)$.
By Lemma \ref{wc_sol} we infer  that
$u_{v_{n_k}} \to u_v$   
in $ C([0,T], H)
  \bigcap L^2(0,T; V) $, 
  which concludes the proof.
 \end{proof}

 \begin{lem}\label{gep}
  Suppose   \eqref{F1}-\eqref{F4}
  and \eqref{sig0}-\eqref{sig3}
  hold, $T>0$    and
 $v\in \cala_N$
for some  $N<\infty$.
If $u^\eps_v
=
 \calg^\eps
\left (
W +\eps^{-\frac 12}\int_0^\cdot
v (t) dt
\right )$, then
$u^\eps_v$ is 
the unique solution to  
 \be\label{gep 1}
  du_v^{\eps}  
  + (-\Delta)^ \alpha  u_v^{\eps}   dt
  + F(t, \cdot, u_v^{\eps} ) dt
  =  
    g(t)   dt
    +  \sigma (t,   u_v^{\eps} )    v dt
  +\sqrt{\eps} \sigma (t,   u_v^{\eps} )    {dW} ,
  \ \ t\in (0,T),
  \ee  
 with  initial condition 
 $
 u^\eps_{v} (0)=u_0\in H. $

  Furthermore,
     for each $R>0$  
    there  exists
       $M_3 =M_3  (R, T,N)>0$ such that  
       for any $u_{0} 
       \in H$ with 
    $\| u_{0} \|\le R  $   and any
    $v \in \cala_N$,
    the solution $u_{v}^\eps$    
    satisfies for all $\eps\in (0,1)$,
  \be\label{gep 2}
 \E\left (
 \| u^\eps_v \|^2_{C([0,T], H)}
 +  \| u^\eps_v \|^2_{L^2(0,T;V )}
 + \| u^\eps_v \|^p_{L^p(0,T; L^p(\R^n))}
 \right )
 \le M_3.
 \ee
 \end{lem}

 \begin{proof}
 By Girsanov\rq{}s theorem,
 we find that
for every $\eps>0$ and
  $v\in \cala_N$, 
$\widetilde{W}
=
W +\eps^{-\frac 12}\int_0^\cdot
v (t) dt$
is a 
cylindrical Wiener
process in $l^2$  
under the probability
$\widetilde{P}^\eps_{v}$:
$$
 {\frac {d \widetilde{P}^\eps_{v}}{dP}}
 =
 \exp
 \left \{
 -\eps^{-\frac 12}
 \int_0^T v
 (t) dW
 -{\frac 12}\eps^{-1}
 \int_0^T \| v(t) \|_H^2 dt
 \right \} ,
 $$
 and hence  
 $u_{v}^\eps
 =\calg^\eps
 (\widetilde{W})
  $
is the unique solution of 
\eqref{intr1}-\eqref{intr3}
with $W$ replaced by
$\widetilde{W}$, which
implies that 
$u^\eps_{v}$ is the unique solution of
\eqref{gep 1} with initial condition
$
 u^\eps_{v} (0)=u_0$.

Next, we prove  \eqref{gep 2}.
By \eqref{gep 1} and Ito\rq{}s rule we have,
for $t\in [0,T]$,
$$
\| u_{v}^\eps(t) \|^2
+2\int_0^t
\| (-\Delta)^{\frac {\alpha}2}  u_{v}^\eps
(s) \|^2 ds
+
2\int_0^t\int_{\R^n}
 F(s,x,  u_{v}^\eps (s) )  u_{v}^\eps (s) ds
 $$
 $$
 =
 \| u_0 \|^2
 +2\int_0^t ( u_{v}^\eps(s), g(s)) ds
 +2\int_0^t
 (u_{v}^\eps (s), \sigma (s, u_{v}^\eps (s)) v(s)) ds
 $$
\be\label{gep 3}
 +\eps\int_0^t
 \| \sigma (s, u_{v}^\eps (s))\|^2_{\call_2(l^2, H)} ds
 + 2\sqrt{\eps}
 \int_0^t (u_{v}^\eps (s),
 \sigma (s, u_{v}^\eps (s)) dW).
 \ee
   By  \eqref{F2} we have
    \be\label{gep 4}
 2 \int_{\R^n} F(s,x,  u_{v}^\eps (s))   u_{v}^\eps (s)  dx 
  \ge  2  \lambda_1 \|  u_{v}^\eps (s)\|^p_{L^p(\R^n)}
  - 2\|\psi_1 (t)\|_{L^1(\R^n)}.
  \ee 
    By Young\rq{}s inequality we have
 \be\label{gep 5} 
 2  |(  u_{v}^\eps (s), g(s) )|
  \le    \| u_{v}^\eps (s) \|^2 +
    \| g(s \|^2.
  \ee
  Similar to \eqref{exis_solp5} we have 
  for  $s\in [0,T]$,
 $$
  2 | ( u_{v}^\eps (s),   \sigma (s,  u_{v}^\eps (s)) v(s) )|
   $$
   \be\label{gep 6}
  \le  
   \left ( 1
+ L_1 \| v(s) \|^2_{l^2}
\right )   \| u_{v}^\eps (s) \|^2
+ \left (   L_1 + 2
\|\sigma_1\|^2_{
C([0,T], L^2(\R^n, l^2))
}
\right )
  \| v(s)\|^2_{l^2} .
 \ee 
  By \eqref{sig6} we have
  for $\eps\in (0,1)$ and $s\in [0, T]$,
    \be\label{gep 7}
  \eps 
 \| \sigma (s, u_{v}^\eps (s))\|^2_{\call_2(l^2, H)}  
 \le
 \eps L_1 ( 1+   \| u_{v}^\eps (s) \|^2)
 +2 \eps
 \|\sigma_1\|^2_{
C([0,T], L^2(\R^n, l^2))}.
\ee
    By \eqref{gep 3}-\eqref{gep 7} we get
   for $\eps\in (0,1)$ and $t\in [0, T]$,
    $$
\| u_{v}^\eps(t) \|^2
+2\int_0^t
\| (-\Delta)^{\frac {\alpha}2}  u_{v}^\eps
(s) \|^2 ds
+
2\lambda_1
\int_0^t \| u_{v}^\eps (s) \|^p_{L^p(\R^n)} ds
 $$
 $$
 \le
 \| u_0 \|^2
 + c_1\int_0^t
 (1+ \| v(s)\|^2_{l^2}) \| u_{v}^\eps (s) \|^2ds
 +c_1\int_0^t \| v(s) \|^2_{l^2} ds
 +c_1t + \int_0^t
 \| g(s) \|^2 ds
 $$
  \be\label{gep 10}
  + 2 \int_0^t \| \psi_1 (s) \|_{L^1(\R^n)} ds 
 + 2 
 \int_0^t (u_{v}^\eps (s),
 \sigma (s, u_{v}^\eps (s)) dW),
 \ee
 where $c_1=c_1(T)>0$.
 
   If $u_0\in H$ with $\| u_0\|\le R$
  and $v\in \cala_N$, then by
  \eqref{gep 10} we get
  for all $t\in [0,T]$, 
  $$
 \| u_{v}^\eps(t) \|^2
+2\int_0^t
\| (-\Delta)^{\frac {\alpha}2}  u_{v}^\eps
(s) \|^2 ds
+
2\lambda_1
\int_0^t \| u_{v}^\eps (s) \|^p_{L^p(\R^n)} ds
 $$
  \be\label{gep 11} 
 \le c_2   
 + c_1\int_0^t
 (1+ \| v(s)\|^2_{l^2}) \| u_{v}^\eps (s) \|^2ds
 + M(t) ,
\ee
 where 
 $$
  M(t) = 2 
 \int_0^t (u_{v}^\eps (s),
 \sigma (s, u_{v}^\eps (s)) dW), 
 \quad
 c_2
 =R^2
   +c_1N
 +c_1T+ \|g\|^2_{L^2(0,T; H)}
 +   2  \| \psi_1  \|_{L^1(0,T; L^1(\R^n) )}.
 $$
 
 By Gronwall\rq{}s inequality, we get
 from \eqref{gep 11} that
  for all $t\in [0,T]$, 
  $$
 \| u_{v}^\eps(t) \|^2
+2\int_0^t
\| (-\Delta)^{\frac {\alpha}2}  u_{v}^\eps
(s) \|^2 ds
+
2\lambda_1
\int_0^t \| u_{v}^\eps (s) \|^p_{L^p(\R^n)} ds
 $$
$$
 \le c_2 
 e^{c_1
 \int_0^t
 (1+ \| v(r)\|^2_{l^2}) dr
 }  
 + 
 \int_0^t
 e^{ c_1
 \int_s^t
 (1+ \| v(r)\|^2_{l^2}) dr
 }
 dM(s)   
$$
$$
 \le c_2 
 e^{c_1
 \int_0^t
 (1+ \| v(r)\|^2_{l^2}) dr
 }  
 + 
 M(t)
 +c_1 \int_0^t
  (1+ \| v(s)\|^2_{l^2})
 e^{ c_1
 \int_s^t
 (1+ \| v(r)\|^2_{l^2}) dr
 }
 M(s) ds  
$$
$$
 \le c_2 
 e^{c_1
 \int_0^t
 (1+ \| v(r)\|^2_{l^2}) dr
 }  
 + 
 M(t)
 +c_1
 (\sup_{0\le s \le t}
 |M(s)|) \  e^{ c_1
 \int_0^t
 (1+ \| v(r)\|^2_{l^2}) dr
 }
  \int_0^t
  (1+ \| v(s)\|^2_{l^2})
    ds  
$$
$$
 \le c_2 
 e^{c_1 
 (T+ N)
 }  
 + 
 M(t)
 +c_1
 (\sup_{0\le s \le t}
 |M(s)|) \   e^{c_1 
 (T+ N)
 } (T+N)  
 $$
  \be\label{gep 12} 
\le c_3 + c_3(\sup_{0\le s \le t}
 |M(s)|),
 \ee
 where $c_3=c_3(R, T, N)>0$,
 By \eqref{gep 12} we obtain,
  for all $t\in [0,T]$, 
  $$
 \E
 \left (
 \sup_{0\le r\le t}
 \left ( 
 \| u_{v}^\eps(r) \|^2
+2\int_0^r
\| (-\Delta)^{\frac {\alpha}2}  u_{v}^\eps
(s) \|^2 ds
+
2\lambda_1
\int_0^r \| u_{v}^\eps (s) \|^p_{L^p(\R^n)} ds
\right )
\right )
 $$
  \be\label{gep 13} 
\le c_3 + c_3
\E \left (\sup_{0\le r \le t}
 |M(r)|
 \right ). 
 \ee
 By
 \eqref{sig6}  and 
  the Burkholder inequality, we have
 $$
 c_3 \E \left (\sup_{0\le r \le t}
 |M(r)|
 \right )
 = 2  c_3 \E \left (\sup_{0\le r \le t}
 \left | \int_0^r (u_{v}^\eps (s),
 \sigma (s, u_{v}^\eps (s) )  dW)
 \right | 
 \right )
 $$
 $$
 \le
  6 c_3 \E \left ( 
  \left (\int_0^t
  \| u_{v}^\eps (s)\|^2
 \| \sigma (s, u_{v}^\eps (s)  )\|^2_{\call_2(l^2, H)}
 ds 
 \right )^{\frac 12}
 \right )
 $$
  $$
 \le
  6 c_3  \E \left ( 
  \sup_{0\le s \le t}
  \| u_{v}^\eps (s)\| 
  \left (\int_0^t
   \| \sigma (s, u_{v}^\eps (s)  )\|^2_{\call_2(l^2, H)}
 ds 
 \right )^{\frac 12}
 \right )
 $$
 $$
 \le
 {\frac 12}
 \E \left (
  \sup_{0\le s \le t}
  \| u_{v}^\eps (s)\| ^2
 \right )
 +18c_3^2
   \E \left (  
  \int_0^t
   \| \sigma (s, u_{v}^\eps (s)  )\|^2_{\call_2(l^2, H)}
 ds  
 \right )
 $$
 $$
 \le
 {\frac 12}
 \E \left (
  \sup_{0\le s \le t}
  \| u_{v}^\eps (s)\| ^2
 \right )
 +18c_3^2 
   \E \left (  
  \int_0^t
 \left ( L_1 +L_1  \| u_{v}^\eps (s)\| ^2 
  +2\| \sigma_1\|^2
  _{
  C([0,T], L^2(\R^n, l^2 ))
  }
  \right )
 ds  
 \right )
 $$
\be\label{gep 14}
 \le
 {\frac 12}
 \E \left (
  \sup_{0\le s \le t}
  \| u_{v}^\eps (s)\| ^2
 \right )
 +18 c_3^2 T( L_1 + 2 \| \sigma_1\|^2
  _{
  C([0,T], L^2(\R^n, l^2 ))
  } )+ 18c_3^2 L_1  
  \int_0^t  \E(   \| u_{v}^\eps (s)\| ^2 ) 
 ds  .
 \ee
 It follows from
 \eqref{gep 13}-\eqref{gep 14} that
 for all $t\in [0,T]$, 
  $$
  \E
 \left (
 \sup_{0\le r\le t}
 \left ( 
 \| u_{v}^\eps(r) \|^2
+2\int_0^r
\| (-\Delta)^{\frac {\alpha}2}  u_{v}^\eps
(s) \|^2 ds
+
2\lambda_1
\int_0^r \| u_{v}^\eps (s) \|^p_{L^p(\R^n)} ds
\right )
\right )
 $$
$$
 \le
  2c_3  
 +36 c_3^2 T( L_1 + 2 \| \sigma_1\|^2
  _{
  C([0,T], L^2(\R^n, l^2 ))
  } )+ 36c_3^2 L_1  
  \int_0^t  \E( \sup_{0\le r\le s}
    \| u_{v}^\eps (r)\| ^2 ) 
 ds  ,
$$
which along with Gronwall\rq{}s inequality 
concludes the proof. 
 \end{proof}

 We now prove
 the convergence of solutions
 of \eqref{intr1}-\eqref{intr3} as
 $\eps \to 0$.

 \begin{lem}\label{cso}
  Suppose   \eqref{F1}-\eqref{F4}
  and \eqref{sig0}-\eqref{sig3}
  hold, $T>0$    
   and
 $\{v^\eps\}  \subseteq \cala_N$
for some  $N<\infty$.
If $\{v^\eps\} $
converges in distribution
to $v$ as $S_N$-valued random variables,
then
$ 
\calg^\eps
\left (
W +\eps^{-\frac 12}\int_0^\cdot
v^\eps (t) dt
\right )$
converges
to $\calg^0\left (
\int_0^\cdot
v(t) dt
\right )$
in $ C([0,T], H) \bigcap L^2(0,T; V)$
in distribution.
\end{lem}

\begin{proof}
Let 
 $u_{v^\eps}^\eps 
 =
 \calg^\eps
\left (
W +\eps^{-\frac 12}\int_0^\cdot
v^\eps (t) dt
\right )$. 
By Lemma 
\ref{gep}, 
$u_{v^\eps}^\eps $  
satisfies 
\be\label{cso 1}
  d u_{v^\eps}^\eps
  + (-\Delta)^ \alpha   u_{v^\eps}^\eps   dt
  + F(t, \cdot,  u_{v^\eps}^\eps ) dt
  =  
    g(t)   dt
    +  \sigma (t,    u_{v^\eps}^\eps  )    v ^\eps dt
  +\sqrt{\eps} \sigma (t,    u_{v^\eps}^\eps  )    {dW} ,
   \ee  
 with  
 $
  u_{v^\eps}^\eps  (0)=u_0\in H. $ 
 Let $u_v=
  \calg^0\left (
\int_0^\cdot
v(t) dt
\right )   $. Then
  $u_v$
is the solution 
 to \eqref{contr1}-\eqref{contr2}.
 We will show that 
$u^\eps_{v^\eps} $ converges to $u_v$
in $ C([0,T], H) \bigcap L^2(0,T; V)$
in distribution as $\eps \to 0$, 
for which 
we  first  deal with   the convergence
of 
$u^\eps_{v^\eps} - u_{v^\eps}$
with
$u_{v^\eps}=  \calg^0\left (
\int_0^\cdot
v^\eps (t) dt \right ) $.
Note that  $u_{v^\eps}$ satisfies 
 \be\label{cso 2}
{\frac {d u_ {v^\eps}   }{dt}}
+  (-\Delta)^\alpha u_ {v^\eps}   
+F(t, \cdot, u_ {v^\eps}   )=  
  g (t)  
 +   \sigma(t, u_ {v^\eps}  )  
     {v^\eps}   ,
\ee
 with  $
 u_ {v^\eps}  (0)=u_0  $. 
  By \eqref{cso 1}-\eqref{cso 2} we have
   $$
 d  ( u_{v^\eps}^\eps -  u_{v^\eps}  )
 +     (-\Delta)^ \alpha ( u^\eps_{v^\eps}  -  u_{v^\eps  }
 ) dt 
 +
   ( F(t,\cdot, u^\eps_{v^\eps}  ) -  
   F(t, \cdot, u_{v^\eps}  ) ) dt   
$$
  \be\label{cso 3}
  =
\left ( 
 \sigma(t, u^\eps_{v^\eps}  )  v^\eps
 -
  \sigma(t, u_{v^\eps}  )  v^\eps
\right )  dt
 + \sqrt{\eps} \sigma(t, u^\eps_{v^\eps}  )  
   dW  ,
\ee
with  
$u_{v^\eps}^\eps(0) -  u_{v^\eps} (0)=0$.

By \eqref{cso 3} and It\^{o}\rq{}s formula we obtain
$$
 \|  u_{v^\eps}^\eps(t) -  u_{v^\eps}(t) \|^2
+2 
\int_0^t
\|  (-\Delta)^{\frac {\alpha}2}
( u^\eps_{v^\eps}(s)  -  u_{v^\eps} (s)
 ) \|^2  ds
 $$
 $$
 =- 2
 \int_0^t\int_{\R^n}
  ( F(s,x, u^\eps_{v^\eps}  ) -  F(s,
  x, u_{v^\eps})  )
 (u^\eps_{v^\eps}  -  u_{v^\eps}  ) dx ds
  $$
  $$ 
 +
  2 
  \int_0^t
  \left ( 
 \sigma(s, u^\eps_{v^\eps} (s) )  v^\eps(s)
 -
  \sigma(s, u_{v^\eps} (s)  )  v^\eps (s),
  \  u^\eps_{v^\eps}(s)  -  u_{v^\eps}(s)
\right )  ds
$$
   \be\label{cso 4}
+ \eps
\int_0^t
 \|   \sigma(s, u^\eps_{v^\eps} (s)  )\|
_{\call_2(l^2, H)}^2  ds
+ 
2\sqrt{\eps} \int_0^t
\left (
u^\eps_{v^\eps}(s)  -  u_{v^\eps}(s),\
 \sigma(s, u^\eps_{v^\eps} (s)  ) 
   dW  \right ) .
\ee
 By \eqref{F4} we have
$$
 -  2
 \int_0^t\int_{\R^n}
  ( F(s,x, u^\eps_{v^\eps}  ) -  F(s,
  x, u_{v^\eps})  )
 (u^\eps_{v^\eps}  -  u_{v^\eps}  ) dx ds
 $$
    \be\label{cso 5}
   \le 2\| \psi_3  \|_{
   L^\infty(0,T; L^\infty (\R^n) )}
  \int_0^t \| u^\eps_{v^\eps} 
  (s)  -  u_{v^\eps} (s) \|^2ds.
  \ee
  By \eqref{sig7} we have
$$  
  2 
  \int_0^t
  \left ( 
 \sigma(s, u^\eps_{v^\eps} (s) )  v^\eps(s)
 -
  \sigma(s, u_{v^\eps} (s)  )  v^\eps (s),
  \  u^\eps_{v^\eps}(s)  -  u_{v^\eps}(s)
\right )  ds
$$
$$  
 \le  2 
  \int_0^t
  \|
 \sigma(s, u^\eps_{v^\eps} (s) )  
 -
  \sigma(s, u_{v^\eps} (s)  )
  \|_{\call_2(l^2, H)}
  \| v^\eps (s) \|_{l^2} \|
  u^\eps_{v^\eps}(s)  -  u_{v^\eps}(s)
  \|   ds
$$
   \be\label{cso 6}
  \le
  2 \|\kappa\| _{L^\infty(\R^n) }
  \left ( \sum_{k=1}^\infty
  \alpha_k^2
  \right )^{\frac 12}
  \int_0^t 
  \| v^\eps (s) \|_{l^2} \|
  u^\eps_{v^\eps}(s)  -  u_{v^\eps}(s)
  \|^2   ds.
  \ee

It follows from \eqref{cso 4}-\eqref{cso 6} that
for all $t\in [0,T]$,
$$
 \|  u_{v^\eps}^\eps(t) -  u_{v^\eps}(t) \|^2
+2 
\int_0^t
\|  (-\Delta)^{\frac {\alpha}2}
( u^\eps_{v^\eps}(s)  -  u_{v^\eps} (s)
 ) \|^2  ds
 $$
 $$
\le
 c_1 \int_0^t (1+ \| v^\eps (s)\|_{l^2})
 \|  u^\eps_{v^\eps}(s)  -  u_{v^\eps}(s) \|^2
 ds
 $$
   \be\label{cso 7}
+ \eps
\int_0^t
 \|   \sigma(s, u^\eps_{v^\eps} (s)  )\|
_{\call_2(l^2, H)}^2  ds
+ 
2\sqrt{\eps}  
\int_0^t
\left (
u^\eps_{v^\eps}(s)  -  u_{v^\eps}(s),\
 \sigma(s, u^\eps_{v^\eps} (s)  ) 
   dW  \right ) ,
\ee
 where 
 $c_1
 =2\| \psi_3  \|_{
   L^\infty(0,T; L^\infty (\R^n) )}
   + 
 2 \|\kappa\| _{L^\infty(\R^n) }
  \left ( \sum_{k=1}^\infty
  \alpha_k^2
  \right )^{\frac 12}$.

  Given $R>0$  and $\eps\in (0,1)$, define 
$$
 \tau^\eps_R 
 =\inf \left \{
 t\ge 0: \| u^\eps_{v^\eps} (t) \| \ge R
 \right \}\wedge T.
$$
By \eqref{cso 7} we have 
for all $t\in [0,T]$,
 $$ \sup_{0\le r \le t}
 \left (
 \|  u_{v^\eps}^\eps(r \wedge \tau^\eps_R ) 
 -  u_{v^\eps}( r \wedge \tau^\eps_R) \|^2
+2 
\int_0^{r \wedge \tau^\eps_R}
\|  (-\Delta)^{\frac {\alpha}2}
( u^\eps_{v^\eps}(s)  -  u_{v^\eps} (s)
 ) \|^2  ds
 \right )
 $$
 $$
\le
 c_1 \int_0^{t \wedge \tau^\eps_R} (1+ \| v^\eps (s)\|_{l^2})
 \|  u^\eps_{v^\eps}(s)  -  u_{v^\eps}(s) \|^2
 ds
 $$
$$
+ \eps
\int_0^{t \wedge \tau^\eps_R}
 \|   \sigma(s, u^\eps_{v^\eps} (s)  )\|
_{\call_2(l^2, H)}^2  ds
+ 
2\sqrt{\eps}
\sup_{0\le r \le t}
 \left|
\int_0^{r \wedge \tau^\eps_R}
\left (
u^\eps_{v^\eps}(s)  -  u_{v^\eps}(s),\
 \sigma(s, u^\eps_{v^\eps} (s)  ) 
   dW  \right ) 
   \right |
   $$
   $$
\le
 c_1 \int_0^{t  } (1+ \| v^\eps (s)\|_{l^2})
 \sup_{0\le r \le s}
 \|  u^\eps_{v^\eps}(r\wedge \tau^\eps_R) 
  -  u_{v^\eps}(r\wedge \tau^\eps_R) \|^2
 ds
 $$
   \be\label{cso 8}
+ \eps
\int_0^{T\wedge \tau^\eps_R}
 \|   \sigma(s, u^\eps_{v^\eps} (s)  )\|
_{\call_2(l^2, H)}^2  ds
+ 
2\sqrt{\eps}
\sup_{0\le r \le T}
 \left|
\int_0^{r \wedge \tau^\eps_R}
\left (
u^\eps_{v^\eps}(s)  -  u_{v^\eps}(s),\
 \sigma(s, u^\eps_{v^\eps} (s)  ) 
   dW  \right ) 
   \right |.
 \ee
   
     By \eqref{cso 8} and Gronwall\rq{}s inequality,
  we get
 for all $t\in [0,T]$, $P$-almost surely,
 $$ \sup_{0\le r \le t}
 \left (
 \|  u_{v^\eps}^\eps(r \wedge \tau^\eps_R ) 
 -  u_{v^\eps}( r \wedge \tau^\eps_R) \|^2
+2 
\int_0^{r \wedge \tau^\eps_R}
\|  (-\Delta)^{\frac {\alpha}2}
( u^\eps_{v^\eps}(s)  -  u_{v^\eps} (s)
 ) \|^2  ds
 \right )
 $$
 \be\label{cso 9}
+ \eps c_2
\int_0^{T\wedge \tau^\eps_R}
 \|   \sigma(s, u^\eps_{v^\eps} (s)  )\|
_{\call_2(l^2, H)}^2  ds
+ 
2\sqrt{\eps}c_2
\sup_{0\le r \le T}
 \left|
\int_0^{r \wedge \tau^\eps_R}
\left (
u^\eps_{v^\eps}(s)  -  u_{v^\eps}(s),\
 \sigma(s, u^\eps_{v^\eps} (s)  ) 
   dW  \right ) 
   \right |,
 \ee
 where 
$ 
c_2= 
e^{c_1
(T+   T^{\frac 12} N^{\frac 12}  ) 
 } $.

For the first term on the right-hand side of 
  \eqref{cso 9},   
  by \eqref{sig6}  we get 
  $$
\int_0^ { T \wedge \tau^\eps_R }
 \|   \sigma(s, u^\eps_{v^\eps} (s)  )\|
_{\call_2(l^2, H)}^2  ds
\le
\int_0^ { T \wedge \tau^\eps_R }
( L_1 +L_1
\| u^\eps_{v^\eps} (s) \|^2
+ 2\| \sigma_1\|^2_{C([0,T], L^2(\R^n,l^2))}
)
  ds 
$$ 
 \be\label{cso 9a}
  \le  T
 \left ( L_1 +L_1R^2
+ 2\| \sigma_1\|^2_{C([0,T], L^2(\R^n,l^2))}
\right ),
\ee
and hence
 \be\label{cso 10}
  \lim_{\eps \to 0}
  \eps c_2
\int_0^ { T \wedge \tau^\eps_R }
 \|   \sigma(s, u^\eps_{v^\eps} (s)  )\|
_{\call_2(l^2, H)}^2  ds
    =0,\quad \text{P-almost surely}.
\ee
Note that  \eqref{exis_sol 2}  implies  that there
exists $c_3=c_3(T, N)>0$ such that
\be
\label{cso 11}
\sup_{\eps\in (0,1)}\sup_{t\in [0,T]}
\| u_{v^\eps} (t) \| \le c_3.
\ee
 For the last term on the right-hand side of
\eqref{cso 9},
by \eqref{cso 9a},
 \eqref{cso 11} 
and
Doob\rq{}s maximal inequality we  obtain
$$
\E\left (
 \sup_{0\le r \le T}
 \left|
\int_0^{r \wedge \tau^\eps_R} 2\sqrt{\eps}c_2
\left (
u^\eps_{v^\eps}(s)  -  u_{v^\eps}(s),\
 \sigma(s, u^\eps_{v^\eps} (s)  ) 
   dW  \right ) 
   \right |^2
   \right )
$$
$$
\le 16 \eps c_2^2
\E\left ( 
  \int_0^ { T\wedge \tau^\eps_R}
 \|
u^\eps_{v^\eps}(s)  -  u_{v^\eps}(s)\|^2
 \| \sigma(s, u^\eps_{v^\eps} (s)   )\|^2_
 {\call_2(l^2, H )}
 ds 
   \right )
$$
$$
\le 16 \eps c_2^2
(R+c_3)^2
\E\left ( 
  \int_0^ { T\wedge \tau^\eps_R}
 \| \sigma(s, u^\eps_{v^\eps} (s)   )\|^2_{\call_2(
 l^2,H )}
 ds 
   \right )
$$
  \be\label{cso 12}
\le 16\eps  c_2^2 
(R+c_3)^2 
 T
 \left ( L_1 +L_1R^2
+ 2\| \sigma_1\|^2_{C([0,T], L^2(\R^n,l^2))}
\right )
\to 0 \ \text{ as } \ \eps \to 0.
\ee

By \eqref{cso 9},
\eqref{cso 10} and
\eqref{cso 12}, we get 
   \be\label{cso 13}
 \lim_{\eps \to 0}
\  \sup_{0\le t \le T}
  \|  u_{v^\eps}^\eps(t \wedge \tau^\eps_R ) 
 -  u_{v^\eps}( t\wedge \tau^\eps_R) \|^2
   =0
  \quad \text{in probability},
  \ee
  and
    \be\label{cso 14}
 \lim_{\eps \to 0}
  \int_0^{T \wedge \tau^\eps_R}
\|   
( u^\eps_{v^\eps}(s)  -  u_{v^\eps} (s)
 ) \|^2_V  ds 
  =0
  \quad \text{in probability}.
\ee

 On the other hand, by
   \eqref{gep 2}  we have
 for all $\eps \in (0,1)$,
   \be\label{cso 15}
 P(\tau^\eps_R <T)
\le 
 P \left (
 \sup_{t\in [0,T]}
 \| u^\eps_{v^\eps} (t) \|  \ge R   
 \right )
 \le
 {\frac 1{R^2}}
 \E \left (
 \sup_{t\in [0,T]}
 \| u^\eps_{v^\eps} (t) \|^2 
 \right )
  \le
 {\frac {M_3}{R^2}},
\ee
 where $M_3=M_3(T, N)>0$.
 It follows  from
 \eqref{cso 15}
 that
 for every $\eta>0$,
 $$
 P 
\left (
\sup_{0\le t \le T}
 \|  u_{v^\eps}^\eps(t )
  -  u_{v^\eps}( t ) \| >\eta
\right )
$$
$$
 \le P 
\left (
\sup_{0\le t \le T}
 \|  u_{v^\eps}^\eps(t )
  -  u_{v^\eps}( t ) \| >\eta,
  \  \tau^\eps_R =T
\right )
+
  P 
\left (
\sup_{0\le t \le T}
 \|  u_{v^\eps}^\eps(t )
  -  u_{v^\eps}( t ) \| >\eta,
  \  \tau^\eps_R <T
\right )
$$
$$
 \le P 
\left (
\sup_{0\le t \le T}
 \|  u_{v^\eps}^\eps(t \wedge  \tau^\eps_R  )
  -  u_{v^\eps}( t \wedge  \tau^\eps_R) \| >\eta  
\right )
+
   {\frac {M_3}{R^2}}.
$$
First letting $\eps \to 0$  and then
 $R\to \infty$, we get
from  \eqref{cso 13} that 
 \be\label{cso 16}
 \lim_{\eps \to 0} 
 ( u_{v^\eps}^\eps
  -   u_{v^\eps}  ) =0 
  \ \text{ in  } {C([0,T],H)}
\  \text{ in probability}.
\ee

Similarly, by   \eqref{cso 15} we have,
for every $\eta>0$,
 $$
 P 
\left (
 \int_0^T
 \|  u_{v^\eps}^\eps(t )
  -  u_{v^\eps}( t ) \|^2_V dt >\eta^2
\right )
$$
$$
 \le P 
\left (
 \int_0^T
 \|  u_{v^\eps}^\eps(t )
  -  u_{v^\eps}( t ) \|^2_V dt >\eta^2,
  \  \tau^\eps_R =T
\right )
+
  P 
\left (
 \int_0^T
 \|  u_{v^\eps}^\eps(t )
  -  u_{v^\eps}( t ) \|^2_V dt >\eta^2,
  \  \tau^\eps_R <T
\right )
$$
$$
 \le P 
\left (
 \int_0^{T \wedge \tau_R^\eps}
 \|  u_{v^\eps}^\eps(t )
  -  u_{v^\eps}( t ) \|^2_V dt >\eta^2 
\right )
+
   {\frac {M_3}{R^2}}.
$$
First letting $\eps \to 0$  and then
 $R\to \infty$, we get
from  \eqref{cso 14} that 
 \be\label{cso 17}
 \lim_{\eps \to 0} 
 ( u_{v^\eps}^\eps
  -   u_{v^\eps}  ) =0 
  \ \text{ in  } { L^2(0,T; V) }
\  \text{ in probability}.
\ee
It follows from \eqref{cso 16}
and \eqref{cso 17} that
\be\label{cso 18}
 \lim_{\eps \to 0} 
 ( u_{v^\eps}^\eps
  -   u_{v^\eps}  ) =0 
  \ \text{ in  } { C([0,T], H) \bigcap L^2(0,T; V) }
\  \text{ in probability}.
\ee

By assumption, we know
  $\{v^\eps\}$ converges
in distribution to $v$ as 
$S_N$-valued random  elements,
and hence 
by Skorokhod\rq{}s theorem,
there exists a probability
space $(\widetilde{\Omega},
\widetilde{\calf}, \widetilde{P})$
and 
$S_N$-valued random variables
$\widetilde{v}^\eps$
and $\widetilde{v}$ on
$(\widetilde{\Omega},
\widetilde{\calf}, \widetilde{P})$
such that the laws of 
$  \widetilde{v}^\eps$ and $
 \widetilde{v} $ are equal to that
 of 
 $  {v^\eps}$ and $
  {v} $ respectively.
  Furthermore, 
  $ \{\widetilde{v}^\eps\}$
 converges to
 $  \widetilde{v} $ 
  almost surely
  in $S_N$.
 It follows from 
  Lemma \ref{wc_sol} that 
  $$
  u_{\widetilde{v} ^\eps}
  \to u_{\widetilde{v}}
  \ \text{in } C([0,T], H) \bigcap L^2(0,T; V)
  \ \text{almost surely},
 $$
and hence it is also convergent  
  in distribution, which implies that
\be\label{cso 19}
  u_{ {v} ^\eps}
  \to u_{ {v}}
  \ \text{in }  C([0,T], H) \bigcap L^2(0,T; V) 
  \ \text{in distribution}.
\ee
By
  \eqref{cso 18}-\eqref{cso 19}  we infer that
$$
  u^\eps_{ {v} ^\eps}
  \to u_{ {v}}
  \ \text{ in }  C([0,T], H) \bigcap L^2(0,T; V) 
  \ \text{in distribution},
$$
as desired. 
 \end{proof}

{\bf Proof of Theorem \ref{main}}.
     Based on Lemmas  \ref{clev}  and \ref{cso},
     we see that Theorem \ref{main} follows from 
  Proposition \ref{LP1}
 immediately.

Under further  conditions on the 
nonlinear drift term,  we will improve Theorem
\ref{main} in the next section, and show the
LDP of 
 \eqref{intr1}-\eqref{intr3}
 in the space
 $C([0,T], H) \bigcap L^2(0,T;V)
 \bigcap L^p(0,T; L^p(\R^n))$.

 \section{Large deviation principles
 under strong dissipativeness  }
  \setcounter{equation}{0}
 
This section is devoted to the LDP
of \eqref{intr1}-\eqref{intr3}
 in  
 $C([0,T], H) \bigcap L^2(0,T;V)
 \bigcap L^p(0,T; L^p(\R^n))$,
 for which we further assume that
  the nonlinearity $F$ satisfies the   strong
  dissipativeness condition:
 for every $t, u_1, u_2\in \R$   and $x\in \R^n$,
 \be\label{Fa}
 \left (
 F(t,x, u_1)-
 F(t,x, u_2)
 \right ) (u_1-u_2)
 \ge \lambda_4 |u_1 -u_2|^p
 -\psi_5 (t,x) |u_1-u_2|^2,
 \ee
 where
 $\lambda_4>0$ is a constant
 and $\psi_5\in L^\infty_{loc} (\R, L^\infty(\R^n))$.

The proof of the LDP of 
\eqref{intr1}-\eqref{intr3}
in the space
 $C([0,T], H) \bigcap L^2(0,T;V)
 \bigcap L^p(0,T; L^p(\R^n))$
 is similar to that
 in the space
  $C([0,T], H) \bigcap L^2(0,T;V) $.
  However, we must improve the convergence
  of  the solutions of the control equation
  \eqref{contr1} 
  as well as the stochastic equation \eqref{gep 1}
  from the space
   $C([0,T], H) \bigcap L^2(0,T;V) $
  to $C([0,T], H) \bigcap L^2(0,T;V) 
  \bigcap L^p(0,T; L^p(\R^n))$.
  
  The following lemma  is an improvement of Lemma
  \ref{wc_sol}.

 \begin{lem}\label{wc_sola}
 Suppose   \eqref{F1}-\eqref{F4},
   \eqref{sig0}-\eqref{sig3} 
   and \eqref{Fa}  hold.
 Let
  $v, v_n \in L^2(0,T; l^2)$ 
  for all  $n\in \N$
  and  $u_v$,  $u_{v_n}$
  be the solutions of 
  \eqref{contr1}-\eqref{contr2}
  corresponding to $v$  and $v_n$,
  respectively.
    If   $v_n \to v$ weakly in 
  $L^2(0,T; l^2)$,
  then
  $ u_{v_n} \to u_v $ strongly
  in  
 $  C([0,T], H)\bigcap L^2(0,T;V)
  \bigcap L^p(0,T; L^p(\R^n))
 $. 
   \end{lem}

   \begin{proof}
       By \eqref{Fa} we have
  $$
    -
     2\int_{\R^n}
      (F(t,x,  u_{v_n} (t)) -F(t,x,  u_v(t))) 
       ( u_{v_n}   (t) - u_v  (t) ) dx
       $$
      \be\label{wc_sola 1}
\le
-2\lambda_4  
\| u_{v_n}   (t) - u_v  (t)\|_{L^p(\R^n)}^p
+
 2\|\psi_5(t) \|_{L^\infty(\R^n)} \|u_{v_n}   (t) - u_v  (t) \|^2.
\ee
    By  \eqref{wc_sol 31},
    \eqref{wc_sol 36} 
    and \eqref{wc_sola 1} we get
         for all $t\in (0, T)$,
      $$
     {\frac d{dt}}
     \|  u_{v_n}   (t) - u_v  (t)\|^2
     +2  \| (-\Delta)^{\frac {\alpha}2}
     ( u_{v_n}   (t) - u_v  (t) )\|^2
     + 2\lambda_4  
\| u_{v_n}   (t) - u_v  (t)\|_{L^p(\R^n)}^p
   $$  
      \be\label{wc_sola 2}
    \le    
2\|\psi_3(t) \|_{L^\infty(\R^n)}
 \|u_{v_n}   (t) - u_v  (t) \|^2
     +
   c_7
     \left (\| v_n(t)\|_{l^2}
     +  \|v (t) \|_{l^2} \right ) 
     \|   u_{v_n}   (t) - u_v  (t)\|.
   \ee
   Then the desired result follows from
   \eqref{wc_sola 2} and the argument of
   \eqref{wc_sol 38}. 
   \end{proof}

 \begin{lem}\label{clev1}
 Suppose   \eqref{F1}-\eqref{F4},
  \eqref{sig0}-\eqref{sig3} 
  and \eqref{Fa} hold.
Then for every $N<\infty$, the set
$$
K_N
=\left \{
\calg^0
\left (
\int_0^\cdot v(t) dt
\right ) : \  v\in S_N
\right \}
$$
is a compact subset
of $ C([0,T], H)
  \bigcap L^2(0,T; V)
    \bigcap L^p(0,T; L^p(\R^n))
  $.
 \end{lem}
 
 \begin{proof}
 The proof is similar to that of Lemma \ref{clev}
 by applying Lemma \ref{wc_sola}.
 \end{proof}

   The next  convergence   is an improvement of Lemma
  \ref{cso}.

 \begin{lem}\label{cso1}
  Suppose   \eqref{F1}-\eqref{F4},
  \eqref{sig0}-\eqref{sig3} and \eqref{Fa}
  hold, $T>0$    
   and
 $\{v^\eps\}  \subseteq \cala_N$
with $N>0$.
If $\{v^\eps\} $
converges in distribution
to $v$ as $S_N$-valued random variables,
then
$ 
\calg^\eps
\left (
W +\eps^{-\frac 12}\int_0^\cdot
v^\eps (t) dt
\right )$
converges
to $\calg^0\left (
\int_0^\cdot
v(t) dt
\right )$
in $ C([0,T], H) \bigcap L^2(0,T; V)
\bigcap L^p(0,T; L^p(\R^n))
$
in distribution.
\end{lem}

\begin{proof}
  By \eqref{Fa} we have
  $$
  - 2
 \int_0^t\int_{\R^n}
  ( F(s,x, u^\eps_{v^\eps}  ) -  F(s,
  x, u_{v^\eps})  )
 (u^\eps_{v^\eps}  -  u_{v^\eps}  ) dx ds
  $$
     \be\label{cso1 1}
\le
-2\lambda_4  \int_0^t
\|  u^\eps_{v^\eps} (s) - u_{v^\eps}   (s)\|_{L^p(\R^n)}^p
ds
+
 2\|\psi_5 \|_{L^\infty(0,T; 
 L^ \infty(\R^n)  )}  \int_0^t
 \| u^\eps_{v^\eps}   (s) - u_{v^\eps}   (s) \|^2 ds.
\ee
 By 
  \eqref{cso 4}, \eqref{cso 6} and 
 \eqref{cso1 1} 
  we get that
for all $t\in [0,T]$,
$$
 \|  u_{v^\eps}^\eps(t) -  u_{v^\eps}(t) \|^2
+2 
\int_0^t
\|  (-\Delta)^{\frac {\alpha}2}
( u^\eps_{v^\eps}(s)  -  u_{v^\eps} (s)
 ) \|^2  ds
 +2\lambda_4  \int_0^t
\|  u^\eps_{v^\eps} (s) - u_{v^\eps}   (s)\|_{L^p(\R^n)}^p
ds
 $$
 $$
\le
 c_1 \int_0^t (1+ \| v^\eps (s)\|_{l^2})
 \|  u^\eps_{v^\eps}(s)  -  u_{v^\eps}(s) \|^2
 ds
 $$
   $$
+ \eps
\int_0^t
 \|   \sigma(s, u^\eps_{v^\eps} (s)  )\|
_{\call_2(l^2, H)}^2  ds
+ 
2\sqrt{\eps}  
\int_0^t
\left (
u^\eps_{v^\eps}(s)  -  u_{v^\eps}(s),\
 \sigma(s, u^\eps_{v^\eps} (s)  ) 
   dW  \right ) ,
$$
The rest of the proof is almost identical to that
of Lemma \ref{cso} and thus omitted.
  \end{proof}

The main result of the paper is stated below.

 \begin{thm}\label{main1}
  Suppose   \eqref{F1}-\eqref{F4},
  \eqref{sig0}-\eqref{sig3}
  and \eqref{Fa}
  hold, $T>0$,    
   and     $u^\eps$
 is the solution of 
   \eqref{intr1}-\eqref{intr3}.
   Then the family
   $\{u^\eps\} $ 
satisfies the LDP
in $C([0,T], H) \bigcap L^2(0,T;V)
\bigcap L^p(0,T; L^p(\R^n))
$, 
  as $\eps \to 0$, 
with  good rate function
  given by \eqref{rate_LS}.
\end{thm}

\begin{proof}
The result follows from Proposition \ref{LP1}
and  Lemmas  \ref{clev1}  and \ref{cso1}.
\end{proof}

\section*{Competing interests and declarations}
I declare that the author  has
 no competing interests as defined by
  Springer, or other interests that might be
   perceived to influence the results and/or 
   discussion reported in this paper.

\end{document}